\documentclass[a4paper,USenglish]{elsarticle}

\usepackage{latexsym}
\usepackage{amsmath,amssymb,amsthm,mathtools}
\usepackage{amsfonts}
\usepackage{hyperref}
\usepackage[linesnumbered,lined,ruled,noend,vlined]{algorithm2e}
\usepackage[appendix=inline]{apxproof}
\usepackage{cleveref}
\usepackage{tabularx,environ}
\usepackage{comment}
\usepackage{url}
\usepackage[full]{complexity}
\usepackage{lscape}
\usepackage{thm-restate}
\usepackage{tikz}

\newcommand{\set}[1]{\{#1\}}
\newcommand{\inset}[2]{\{#1 \colon #2\}}
\newcommand{\size}[1]{|#1|}
\newcommand{\order}[1]{O(#1)}
\newcommand{\comp}[1]{\mu(#1)}

\newcommand{\rest}{\mathbin{\mid}}
\newcommand{\cont}{\mathbin{\slash}}
\newcommand{\xor}{\mathbin{\triangle}}
\renewcommand{\poly}{{\rm poly}}

\newcommand{\mat}{\mathbf M}

\newcommand{\parent}[1]{{\tt par}(#1)}

\newtheorem{lemma}{Lemma}
\newtheorem{theorem}[lemma]{Theorem}
\newtheorem{corollary}[lemma]{Corollary}
\newtheorem{proposition}[lemma]{Proposition}
\newtheorem{definition}[lemma]{Definition}
\newenvironment{claim*}[1]{{\par\noindent\underline{Claim:}\space{#1}}}{}

\makeatletter
\@ifpackageloaded{algorithm2e}{
\SetKwProg{Fn}{Function}{}{}
\SetKwProg{Procedure}{Procedure}{}{}
\SetKwProg{Subprocedure}{Subprocedure}{}{}
\SetKwComment{tcc}{//}{}
\SetKwFunction{Maximum}{Maximum}
\SetKwFunction{RecMaximum}{RecMaximum}
\SetKwFunction{Maximal}{Maximal}
\SetKwFunction{RecMaximal}{RecMaximal}
\SetKwFunction{Ranked}{Ranked}
\SetKwFunction{RecRanked}{RecRanked}
\SetKwFunction{Output}{Output}%
\SetKw{Continue}{continue} 
\SetKwInOut{AlgInput}{Input}
\SetKwInOut{AlgOutput}{Output}
\SetKwInOut{AlgPrecondition}{Pre-conditions}
\SetKwInOut{AlgInvariant}{Invariants}
}{}
\makeatother

\makeatletter

\newcolumntype{\expand}{}
\long\@namedef{NC@rewrite@\string\expand}{\expandafter\NC@find}
\makeatletter

\newcommand{\problemtitle}[1]{\gdef\@problemtitle{#1}}
\newcommand{\probleminput}[1]{\gdef\@probleminput{#1}}
\newcommand{\problemoutput}[1]{\gdef\@problemoutput{#1}}
\NewEnviron{problem}{
  \problemtitle{}\probleminput{}\problemoutput{}
  \BODY
  \par\addvspace{.5\baselineskip}
  \noindent{
    \framebox[0.98\textwidth][c]{
        \begin{tabularx}{0.98\textwidth}{@{\hspace{.2\parindent}} l X c}
            \multicolumn{2}{@{\hspace{.2\parindent}}l}{{\sc \@problemtitle}} \\
            \textbf{Input:}  & \@probleminput \\
            \textbf{Output:} & \@problemoutput
        \end{tabularx}
        }
      }
  \par\addvspace{.5\baselineskip}
}
\makeatother

\usepackage{natbib}

\usepackage[mathlines]{lineno}
\newcommand*\patchAmsMathEnvironmentForLineno[1]{
  \expandafter\let\csname old#1\expandafter\endcsname\csname #1\endcsname
  \expandafter\let\csname oldend#1\expandafter\endcsname\csname end#1\endcsname
  \renewenvironment{#1}
     {\linenomath\csname old#1\endcsname}
     {\csname oldend#1\endcsname\endlinenomath}}
\newcommand*\patchBothAmsMathEnvironmentsForLineno[1]{
  \patchAmsMathEnvironmentForLineno{#1}
  \patchAmsMathEnvironmentForLineno{#1*}}
\AtBeginDocument{
\patchBothAmsMathEnvironmentsForLineno{equation}
\patchBothAmsMathEnvironmentsForLineno{align}
\patchBothAmsMathEnvironmentsForLineno{flalign}
\patchBothAmsMathEnvironmentsForLineno{alignat}
\patchBothAmsMathEnvironmentsForLineno{gather}
\patchBothAmsMathEnvironmentsForLineno{multline}
}

\journal{Information and Computation}
\begin{document}

\begin{frontmatter}



\title{Polynomial-Delay Enumeration of Large Maximal Common Independent Sets in Two Matroids and Beyond\tnoteref{grants}}

\tnotetext[grants]{A preliminary version of this paper has been published in \cite{kobayashi_et_al:LIPIcs.MFCS.2023.58}. This work is partially supported by JSPS KAKENHI Grant Numbers JP20H00595, JP23H03344, JP21K17812, JP22H03549, and JP22K17849, JST CREST Grant Number JPMJCR18K3, and JST ACT-X Grant Number JPMJAX2105.}





\author[HU]{Yasuaki Kobayashi}

\affiliation[HU]{organization={Hokkaido University},
            addressline={Kita 14-jo Nishi 9-choume, Kita-ku}, 
            city={Sapporo},
            postcode={060-0814}, 
            state={Hokkaido},
            country={Japan}}

\author[NU]{Kazuhiro Kurita}
\author[HouseiU]{Kunihiro Wasa}

\affiliation[NU]{organization={Nagoya University}, 
            addressline={Furocho, Chikusa-ku}, 
            city={Nagoya},
            postcode={464-8601}, 
            state={Aichi},
            country={Japan}}

\affiliation[HouseiU]{organization={Hosei University},
            addressline={Kajinocho}, 
            city={Koganei},
            postcode={184-8584}, 
            state={Tokyo},
            country={Japan}}

\begin{keyword}
    Polynomial-delay enumeration \sep Ranked enumeration \sep Matroid intersection \sep Matroid matching \sep Reverse search
\end{keyword}

\begin{abstract}
    Finding a \emph{maximum} cardinality common independent set in two matroids (also known as \textsc{Matroid Intersection}) is a classical combinatorial optimization problem, which generalizes several well-known problems, such as finding a maximum bipartite matching, a maximum colorful forest, and an arborescence in directed graphs.
    Enumerating all \emph{maximal} common independent sets in two (or more) matroids is a classical enumeration problem.
    In this paper, we address an ``intersection'' of these problems: Given two matroids and a threshold $\tau$, the goal is to enumerate all maximal common independent sets in the matroids with cardinality at least $\tau$.
    We show that this problem can be solved in polynomial delay and polynomial space.
    Moreover, our technique can be extended to a more general problem, which is relevant to \textsc{Matroid Matching}.
    We give a polynomial-delay and polynomial-space algorithm for enumerating all maximal ``matchings'' with cardinality at least $\tau$, assuming that the optimization counterpart is ``tractable'' in a certain sense.
    This extension allows us to enumerate small minimal connected vertex covers in subcubic graphs.
    We also discuss a framework to convert enumeration with cardinality constraints into ranked enumeration.
\end{abstract}


\end{frontmatter}

\section{Introduction}
The bipartite matching problem is arguably one of the most famous combinatorial optimization problems, which asks to find a maximum cardinality matching in a bipartite graph.
This problem can be solved in polynomial time by polynomial-time algorithms for the maximum flow problem.
The problem is naturally generalized for non-bipartite graphs, which is also solvable in polynomial time~\cite{edmonds_1965}.

Another natural generalization of the bipartite matching problem is \textsc{Matroid Intersection}.
In this problem, we are given two matroids $\mat_1 = (S, \mathcal I_1)$ and $\mat_2 = (S, \mathcal I_2)$, where $\mathcal I_1 \subseteq 2^S$ and $\mathcal I_2 \subseteq 2^S$ are the collections of independent sets of $\mat_1$ and $\mat_2$, respectively, and asked to find a maximum cardinality \emph{common independent set} of $\mat_1$ and $\mat_2$, that is, a maximum cardinality set in $\mathcal I_1 \cap \mathcal I_2$.
When both $\mat_1$ and $\mat_2$ are partition matroids, this problem is equivalent to the bipartite $b$-matching problem, which is a generalization of the bipartite matching problem.
A famous matroid intersection theorem~\cite{edmonix1969submodular} shows a min-max formula and also gives a polynomial-time algorithm for \textsc{Matroid Intersection}~\cite{Lawler1975}.

These classical results give efficient algorithms to find a \emph{single} best (bipartite) matching in a graph or common independent set in two matroids.
This type of objective serves as the gold standard in many algorithmic and computational studies.
However, such a single best solution may not be appropriate for real-world problems due to the complex nature of them~\cite{Eppstein:k-best:2016}.

One possible remedy to this issue is to enumerate \emph{multiple} solutions instead of a single best one.
From the viewpoint of enumeration, the problems of enumerating \emph{maximal} or \emph{maximum} (bipartite) matchings and its generalization are studied in the literature~\cite{DBLP:journals/siamcomp/TsukiyamaIAS77,DBLP:conf/spire/ConteGMUV17,10.1007/3-540-45678-3_32,10.1007/3-540-63890-3_11}.
Enumerating maximal independent sets in a graph is one of the best-studied problems in this area and is solvable in polynomial delay and polynomial space~\cite{DBLP:journals/siamcomp/TsukiyamaIAS77,DBLP:conf/spire/ConteGMUV17}.
Due to the equivalence between matchings in a graph and independent sets in its line graph, we can enumerate all maximal matchings in polynomial delay and polynomial space as well.
Moreover, several enumeration algorithms that are specialized to (bipartite) matchings are known~\cite{10.1007/3-540-45678-3_32,10.1007/3-540-63890-3_11}.

Enumeration algorithms for matroids are also frequently studied in the literature~\cite{FUKUDA1995231,Lawler:Lenstra:Rinnooy:SIAM:1980,Boros:MFCS:2002,DBLP:journals/siamdm/KhachiyanBEGM05,DBLP:journals/algorithmica/KhachiyanBBEGM08}.
Lawler et al.~\cite{Lawler:Lenstra:Rinnooy:SIAM:1980} showed that all maximal common independent sets in $k$ matroids can be enumerated in polynomial delay when $k$ is constant.
For general $k$, this problem is highly related to \textsc{Dualization} (or equivalently \textsc{Minimal Transversal Enumeration}), which can be solved in output quasi-polynomial time\footnote{An enumeration algorithm runs in output quasi-polynomial time if it runs in time $N^{(\log N)^c}$, where $c$ is a constant and $N$ is the combined size of the input and output.}~\cite{Boros:MFCS:2002}.
Apart from common independent sets, enumeration problems related to matroids are studied~\cite{10.1007/11841036_41,DBLP:journals/siamdm/KhachiyanBEGM05}, such as minimal multiway cuts~\cite{DBLP:journals/siamdm/KhachiyanBEGM05} and minimal Steiner forests in graphs~\cite{10.1007/11841036_41}.

In this paper, we consider an ``intersection'' of the above two worlds, optimization and enumeration, for \textsc{Matroid Intersection}.
More specifically, given two matroids $\mat_1$ and $\mat_2$ and an integer $\tau$, we consider the problem of enumerating all \emph{maximal} common independent sets of $\mat_1$ and $\mat_2$ \emph{with cardinality at least $\tau$}.
We refer to this problem as \textsc{Large Maximal Common Independent Set Enumeration}.
By setting $\tau = 0$, we can enumerate all maximal common independent sets of $\mat_1$ and $\mat_2$, and by setting $\tau = {\rm opt}$, we can enumerate all maximum common independent sets of $\mat_1$ and $\mat_2$, where ${\rm opt}$ is the optimal value of \textsc{Matroid Intersection}.
We would like to mention that simultaneously handling these two constraints, maximality and cardinality, would make enumeration problems more difficult (see \cite{Kobayashi:WG:2022,Kobayashi:arXiv:2020,Korhonen:arXiv:2020}, for other enumeration problems).
We show that \textsc{Large Maximal Common Independent Set Enumeration} can be solved in polynomial delay and space.
This extends the results of enumerating maximum common independent sets due to \cite{FUKUDA1995231} and enumerating maximal common independent sets due to \cite{Lawler:Lenstra:Rinnooy:SIAM:1980}.
Thanks to its generality, our enumeration algorithm allows us to enumerate several combinatorial objects with maximality and cardinality constraints, such as bipartite $b$-matchings, colorful forests, and degree-constraint subdigraphs.

To prove this, we devise a reverse search algorithm~\cite{Avis::1996} to enumerate all maximal common independent sets of $\mat_1$ and $\mat_2$.
This algorithm enumerates the solutions in a depth first manner.
To completely enumerate all the solutions without duplicates, we carefully design its search strategy.
We exploit a famous augmenting path theorem for \textsc{Matroid Intersection}~\cite{Lawler1975}.
This enables us to design a ``monotone'' search strategy, yielding a polynomial-delay and polynomial-space enumeration algorithm.
A similar idea is used in \cite{Kobayashi:WG:2022} for \textsc{Large Maximal Matching Enumeration}, that is, enumerating maximal matchings with cardinality at least $\tau$, but we need several nontrivial lemmas to obtain our result.

As a common generalization of a matching in graphs and a common independent set in two matroids, 
a ``matching in a matroid'' is also known.
Let $\mat = (S, \mathcal I)$ be a matroid and $G = (S, E)$ be a graph.
A set of edges $M \subseteq E$ is a \emph{matching} of $(\mat, G)$
if $M$ is a matching in $G$ and $\bigcup_{e \in M}e$ is independent in $\mat$.
The problem of finding a maximum cardinality matching is known as \textsc{Matroid Matching}.
In this paper, we address the problem of enumerating large maximal matchings of a matroid.
However, there is an obstacle to efficiently enumerate them since \textsc{Matroid Matching} cannot be solved in polynomial time when the matroid is given as an independence oracle~\cite{doi:10.1137/0211014}.
Thus, in this paper, we assume that the pair $(\mat, G)$ is \emph{tractable} in a certain sense (which will be defined in \Cref{sec:prelim}).
This assumption would be acceptable since there are many special cases for which $(\mat, G)$ is tractable.
We also give a polynomial-delay and polynomial-space enumeration algorithm for this extension, which includes \textsc{Large Maximal Common Independent Set Enumeration} and \textsc{Large Maximal Matching Enumeration}~\cite{Kobayashi:WG:2022} as special cases.

Finally, we discuss a framework to convert our algorithms into ranked enumeration algorithms.
Although our algorithms enumerate all maximal solutions with cardinality at least $\tau$, 
solutions may not be generated in a sorted order, which is of great importance in database community~\cite{DBLP:conf/pods/RavidMK19,Deep:ICDT:2021}.
A \emph{ranked enumeration algorithm} is an algorithm that enumerates all the solutions in a non-increasing order of their cardinality (or, more generally, objective value).
We discuss how to convert our enumeration algorithms to the ones that enumerate in a ranked manner with a small overhead in the running time.

\section{Preliminaries}\label{sec:prelim}
Let $S$ be a finite set.
We denote the cardinality of $S$ as $n$.
For two sets $X$ and $Y$, the symmetric difference of $X$ and $Y$ is defined as $X \xor Y \coloneqq (X\setminus Y) \cup (Y\setminus X)$.
A pair $\mat = (S, \mathcal I)$ is called a \emph{matroid} if $\mat$ satisfies the following conditions:
\begin{enumerate}
    \item  $\emptyset \in \mathcal I$,
    \item  if $I \in \mathcal I$ and $J\subseteq I$, then $J \in \mathcal I$, and 
    \item  if $I, J \in \mathcal I$ and $\size{I} < \size{J}$, then $I \cup \set{e} \in \mathcal I$ for some $e \in J \setminus I$.
\end{enumerate}
A subset $S'$ of $S$ is called an \emph{independent set} of $\mat$ (or \emph{independent} in $\mat$) if $S'$ is contained in $\mathcal I$ and $S'$ is called a \emph{dependent set} of $\mat$ (or \emph{dependent} in $\mat$) otherwise.
For a subset $X \subseteq S$, the \emph{rank} of $X$ is the maximum cardinality of an independent set that is contained in $X$.
An inclusion-wise maximal independent set of $\mat$ is called a \emph{base} of $\mat$, and  
an inclusion-wise minimal dependent set of $\mat$ is called a \emph{circuit} of $\mat$.
For two distinct circuits $C_1$ and $C_2$ of $M$ with $C_1 \cap C_2 \neq \emptyset$ and $e \in C_1 \cap C_2$, there always exists a circuit $C_3$ of $M$ such that $C_3 \subseteq (C_1 \cup C_2) \setminus\set{e}$.
This property is called the (\emph{weak}) \emph{circuit elimination axiom}~\cite{book:oxley}.
For a matroid $\mat = (S, \mathcal I)$ and a subset $X \subseteq S$, 
the pair $(X, \mathcal J)$ is the \emph{restriction} of $\mat$ to $X$, 
where $\mathcal J=\inset{Y \subseteq X}{Y \in \mathcal I}$.
We denote it as $\mat \rest X$.
Similarly, the pair $(S \setminus X, \mathcal J')$ is the \emph{deletion} of $X$ from $\mat$, where $\mathcal J' = \inset{Y \setminus X}{Y \in \mathcal I}$.
We denote it as $\mat \setminus X$.
Moreover, the pair $(S \setminus X, \mathcal J'')$ is the \emph{contraction} of $X$ from $\mat$, 
where $\mathcal J'' = \inset{Y \subseteq S \setminus X}{ M\rest X \text{ has a base } B \text{ such that } Y \cup B \in \mathcal I}$.
We denote it as $\mat \slash X$.
It is known that for a matroid $\mat = (S, \mathcal I)$ and $X \subseteq S$, $\mat \slash X$, 
$\mat \rest X$, and $\mat \setminus X$ are all matroids~\cite{book:oxley}.
For two matroids $\mat_1 = (S, \mathcal I_1)$ and $\mat_2 = (S, \mathcal I_2)$ defined on the same set $S$, a subset $T \subseteq S$ is a \emph{common independent set} of $\mat_1$ and $\mat_2$ if $T \in \mathcal I_1$ and $T \in \mathcal I_2$.

Let $I_1$ and $I_2$ be distinct independent sets of $\mat$.
In our algorithm, we frequently consider a matroid obtained from $\mat$ by restricting to $I_1 \cup I_2$ and then contracting $I_1 \cap I_2$.
This matroid is defined on $I_1 \xor I_2$ and has some properties shown below.

\begin{proposition}\label{prop:union_intersection}
    Let $I_1$ and $I_2$ be independent sets of $\mat$ and let $\mat' = (M \rest (I_1 \cup I_2)) \cont (I_1 \cap I_2)$.
    $I \subseteq I_1 \xor I_2$ is independent in $\mat'$ if and only if $I \cup (I_1 \cap I_2)$ is independent in $\mat$. 
\end{proposition}
\begin{proof}
    Suppose that $I$ is independent in $\mat'$.
    As $\mat'$ is a contraction of $I_1 \cap I_2$ from $\mat \rest (I_1 \cup I_2)$, $I \cup (I_1 \cap I_2)$ is independent in $\mat \rest (I_1 \cup I_2)$ and hence in $\mat$.
    The converse direction is analogous.
\end{proof}

\begin{proposition}\label{prop:maximal}
    Let $I_1$ and $I_2$ be maximal common independent sets of $\mat_1 = (S, \mathcal I_1)$ and $\mat_2 = (S, \mathcal I_2)$.
    Then, both $I_1\setminus I_2$ and $I_2 \setminus I_1$ are maximal common independent sets of two matroids 
    $\mat_1' = (\mat_1 \rest (I_1 \cup I_2))\cont (I_1 \cap I_2)$ and $\mat_2' = (\mat_2 \rest (I_1 \cup I_2))\cont (I_1 \cap I_2)$.
\end{proposition}
\begin{proof}
    By symmetry, it suffices to show that $I_1 \setminus I_2$ is a maximal common independent set of $\mat'_1$ and $\mat'_2$.
    By~\Cref{prop:union_intersection}, $(I_1 \setminus I_2) \cup (I_1 \cap I_2) = I_1$ is independent in $\mat_1$ if and only if $I_1 \setminus I_2$ is independent in $\mat'_1$.
    Similarly, $I_1$ is independent in $\mat_2$ if and only if $I_1 \setminus I_2$ is independent in $\mat'_2$.
    Thus, $I_1 \setminus I_2$ is a common independent set of $\mat'_1$ and $\mat'_2$.
    To see the maximality, suppose that there is $e \in I_2 \setminus I_1$ such that $(I_1 \setminus I_2) \cup \set{e}$ is a common independent set of $\mat'_1$ and $\mat'_2$.
    By~\Cref{prop:union_intersection}, $I_1 \cup \set{e}$ is a common independent set of $\mat_1$ and $\mat_2$, contradicting the maximality of $I_1$.
\end{proof}

We next define some notations for (directed) graphs.
In this paper, we assume that (directed) graphs have no self-loops.

Let $G = (V, E)$ be an undirected graph.
For $X \subseteq V$, we denote by $G[X]$ the subgraph of $G$ induced by $X$, and for $F \subseteq E$, by $G - F$ the subgraph of $G$ obtained by deleting all edges in $F$.

Let $D = (V, A)$ be a directed graph.
We say that a vertex $v$ is an \emph{out-neighbor} of $u$ ($u$ is an \emph{in-neighbor} of $v$) in $D$ if $D$ has an arc $(u, v)$. 
The set of out-neighbors of $v$ is denoted by $N^+(v)$, and the set of in-neighbors of $v$ is denoted by $N^-(v)$.
A sequence $(v_1, \dots, v_k)$ of distinct vertices is a \emph{directed path} if there is an arc $(v_i, v_{i+1})$ for $1 \le i < k$.
A directed path $(v_1, \dots, v_k)$ in $D$ is called a \emph{directed path without shortcuts}
if $D$ has no arc from $v_i$ to $v_j$ for any $1 \le i < j \le k$ with $i + 1 < j$.

We measure the time complexity of enumeration algorithms with delay complexity~\cite{Johson:Yanakakis:Papadimitriou:IPL:1988}.
The \emph{delay} of an enumeration algorithm is the maximum time elapsed between two consecutive outputs, including preprocessing and post-processing time.
An enumeration algorithm is called a \emph{polynomial-delay enumeration algorithm} if its delay is upper bounded by a polynomial of the size of an input.
An enumeration algorithm is called an \emph{linear incremental-time enumeration algorithm} if, for any $i \le N$, an algorithm outputs at least $i$ solutions in time $\order{i\cdot \poly(n)}$, where $N$ is the number of solutions~\cite{capelli_et_al:LIPIcs.STACS.2023.18}.

Now, we formally define our problems.
Throughout the paper, we assume that matroids are given as \emph{independence oracles}, that is, for a matroid $\mat = (S, \mathcal I)$, we can test whether a subset $X \subseteq S$ belongs to $\mathcal I$ by accessing an oracle for $\mat$.
Moreover, we assume that independence oracles can be evaluated in $Q$ time and in $\hat{Q}$ space.
We say that an enumeration algorithm runs in polynomial delay (resp. polynomial space) if the delay (resp. space) is upper bounded by a polynomial in $n + Q$ (resp. $n + \hat{Q}$).

\begin{definition}
    Given two matroids $\mat_1 = (S, \mathcal I_1)$ and $\mat_2 = (S, \mathcal I_2)$ represented by independence oracles and an integer $\tau$, \textsc{Large Maximal Common Independent Set Enumeration} asks to enumerate all maximal common independent sets of $\mat_1$ and $\mat_2$ with cardinality at least $\tau$.
\end{definition}

\begin{definition}
    Given two matroids $\mat_1 = (S, \mathcal I_1)$ and $\mat_2 = (S, \mathcal I_2)$ represented by independence oracles, \textsc{Ranked Maximal Common Independent Set Enumeration} asks to enumerate all maximal common independent sets of $\mat_1$ and $\mat_2$ in a non-increasing order with respect to cardinality. 
\end{definition}

Let $\mat = (S,\mathcal I)$ be a matroid and $G = (S, E)$ be a graph.
A set of edges $M \subseteq E$ is called a \emph{matching} of $(\mat, G)$
if $M$ is a matching in $G$ and $\bigcup_{e \in M}e$ is independent in $\mat$.
A matching $M$ of $(\mat, G)$ is called a \emph{maximal matching} if for any edge $f \in E \setminus M$,
$M \cup \set{f}$ is not a matching in $G$ or a dependent set of $\mat$.
As noted before, finding a maximum matching of $(\mat, G)$ cannot be solved in polynomial time under the independence oracle model~\cite{doi:10.1137/0211014}.
To design a polynomial-delay algorithm, we focus on ``tractable'' instances.
We say that $(\mat, G)$ is \emph{tractable} if it satisfies the following condition: There are algorithms for finding maximum matchings of $(\mat \rest X, H)$, $(\mat\setminus X, H')$, and $(\mat \cont X, H')$ for any $X \subseteq S$, spanning subgraph $H$ of $G[X]$, and spanning subgraph $H'$ of $G[S\setminus X]$ that run in time $\poly(n + Q)$ and space $\poly(n + \hat{Q})$.
Tractable instances would still be interesting as many natural problems can be reduced to the problem of finding a maximum matching of a tractable matroid.

Clearly, the set of all matchings in $G = (V, E)$ can be encoded to a tractable pair $(\mat_\text{free}, G)$, where $\mat_\text{free} = (V, 2^V)$ is a free matroid defined on $V$.
In particular, if $\mat$ is a linear matroid that is given as a matrix over a field $\mathbb F$, the problem of finding a maximum matching of $(\mat, G)$ is polynomial time~\cite{Gabow1986}.
From a matrix representation of $\mat$, we can compute matrix representations of $\mat \rest X$, $\mat \setminus X$, and $\mat \cont X$ over the same field in polynomial time as well~\cite{book:oxley}, meaning that $(\mat, G)$ is tractable when $\mat$ is given as a matrix over some field $\mathbb F$.
Furthermore, the collection of common independent sets of two matroids $\mat_1$ and $\mat_2$ can be encoded to a pair $(\mat_{1,2}, G_{1,2})$, where $\mat_{1,2} = (S_{1,2}, \mathcal I_{1,2})$ is a matroid such that $S_{1,2} = \{e_1 : e \in S\} \cup \{e_2 : e \in S\}$ and $G_{1,2}$ has an edge between $e_1$ and $e_2$ for $e \in S$. 
It is easy to see that $X \subseteq S$ is a common independent set of $\mat_1$ and $\mat_2$ if and only if $\bigcup_{e \in X}\{e_1, e_2\}$ is a matching of $(\mat_{1,2}, G_{1,2})$.
From the independence oracles of $\mat_1$ and $\mat_2$, we can easily simulate the independence oracles of $\mat \rest X$, $\mat \setminus X$, and $\mat \cont X$ in polynomial time, which means that $(\mat_{1,2}, G_{1,2})$ is tractable using a polynomial-time algorithm for \textsc{Matroid Intersection} (e.g., \cite{Lawler1975}).

In \Cref{sec:parity}, we address the following problem, which simultaneously generalize \textsc{Large Maximal Matching Enumeration} in~\cite{Kobayashi:WG:2022} and \textsc{Large Maximal Common Independent Set Enumeration}.

\begin{definition}
    Given a tractable pair $(\mat, G)$ with a matroid $\mat = (S, \mathcal I)$ represented by an independence oracle and a graph $G = (S, E)$ and an integer $\tau$, 
    \textsc{Large Maximal Matroid Matching Enumeration} asks to enumerate all maximal matchings of $(\mat, G)$ with cardinality at least $\tau$.
\end{definition}

We also discuss the ranked enumeration variant of \textsc{Large Maximal Matroid Matching Enumeration}.

\begin{definition}
    Given a tractable pair $(\mat, G)$ with a matroid $\mat = (S, \mathcal I)$ represented by an independence oracle and a graph $G = (S, E)$, 
    \textsc{Ranked Maximal Matroid Matching Enumeration} asks to enumerate all maximal matchings of $(\mat, G)$ in non-increasing order with respect to cardinality.
\end{definition}

\subsection{Overview of an algorithm for finding a maximum common independent set}\label{subsec:maximum}
Our proposed algorithm for \textsc{Large Maximal Common Independent Set Enumeration} leverages a well-known property used in an algorithm for finding a maximum common independent set in two matroids.
In this paper, we refer to a particular algorithm given by Lawler~\cite{Lawler1975}.

Let $\mat_1 = (S, \mathcal I_1)$ and $\mat_2 = (S, \mathcal I_2)$ be matroids. 
In Lawler's algorithm~\cite{Lawler1975}, we start with an arbitrary common independent set $I$ of $\mat_1$ and $\mat_2$ (e.g., $I \coloneqq \emptyset$), update $I$ to a larger common independent set $I'$ in some ``greedy way''. 
This update procedure is based on the following auxiliary directed graph $D_{\mat_1, \mat_2}(I) = (S \cup \set{s,t}, A)$.

Let $I \subseteq S$ be a common independent set of $\mat_1$ and $\mat_2$.
The set $A = A_1 \cup A_2 \cup A_3 \cup A_4$ of arcs in $D_{\mat_1, \mat_2}(I)$ consists of the following four types of arcs. 
The first type of arcs is defined as 
\begin{align*}
    A_1 = \inset{(e, f)}{e \in I, f \in S \setminus I, I \cup \set{f} \notin \mathcal I_1, (I \cup \set{f}) \setminus \set{e} \in \mathcal I_1},
\end{align*}
that is, an arc $(e, f) \in A_1$ indicates that $I \xor \set{e, f}$ is independent in $\mat_1$. 
Symmetrically, the second type of arcs is defined as
\begin{align*}
    A_2 = \inset{(f, e)}{e \in I, f\in  S\setminus I, I \cup\set{f} \notin \mathcal I_2, (I \cup \set{f}) \setminus \set{e} \in \mathcal I_2},
\end{align*}
that is, an arc $(f, e)$ indicates that $I \xor \set{e, f}$ is independent in $\mat_2$. 
The third and fourth types of arcs are defined as 
\begin{align*}
    A_3 &=  \inset{(s, f)}{f \in S \setminus I, I \cup \set{f} \in \mathcal I_1}\\
    A_4 &= \inset{(f, t)}{f \in S \setminus I, I \cup \set{f} \in \mathcal I_2},
\end{align*}
respectively.
Arcs $(s, f)$ and $(f, t)$ indicate that $I \cup \set{f}$ is independent in $\mat_1$ and in $\mat_2$, respectively.
We illustrate a concrete example of $D_{\mat_1, \mat_2}(I)$ in \Cref{fig:opt}.
In the following, we simply write $D(I)$ to denote $D_{\mat_1, \mat_2}(I)$.

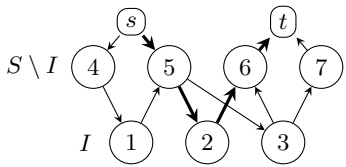
\begin{figure}[t]
    \centering
    \begin{tikzpicture}[every node/.style={draw, circle}]\small
        \newcount\c
        \newcount\l
        \newcount\a
        \c=4
        \node[rectangle, draw=black!0] at (1.7, 1) {$S \setminus I$};
        \node[rectangle, draw=black!0] at (2.4, 0) {$I$};
        \node[rectangle, rounded corners] (s) at (3, 1.6) {$s$};
        \node[rectangle, rounded corners] (t) at (5, 1.6) {$t$};
        \foreach \x [count=\i] in {1, ..., 3} {
            \node (\x) at (\c+\i-2, 0) {${\x}$};
        }    
        \foreach \x [count=\i] in {4, ..., 7} {
            \node (\x) at (\c+\i-2.5, 1) {${\x}$};
        }    
        \draw[->,>=stealth, very thick] (s) -- (5);
        \draw[->,>=stealth] (s) -- (4);
        \draw[->,>=stealth, very thick] (6) -- (t);
        \draw[->,>=stealth] (7) -- (t);
        \draw[->,>=stealth] (1) -- (5);
        \draw[->,>=stealth, very thick] (5) -- (2);
        \draw[->,>=stealth] (5) -- (3);
        \draw[->,>=stealth] (3) -- (7);
        \draw[->,>=stealth] (3) -- (6);
        \draw[->,>=stealth] (4) -- (1);
        \draw[->,>=stealth, very thick] (2) -- (6);
    \end{tikzpicture} 
    \caption{
    This figure depicts an example of the auxiliary graph $D_{\mat_1, \mat_2}(\set{1,2,3})$.
    Let $\mat_1$ and $\mat_2$ be matroids with the same ground set $\set{1, \dots, 7}$ that defined by 
    five bases $\set{1,2,3,4}$, $\set{1,2,3,5}$, $\set{1,3,5,6}$, $\set{1,2,5,6}$, $\set{1,2,5,7}$ and
    six bases $\set{1,2,3,6}$, $\set{1,2,3,7}$, $\set{1,2,5,6}$, $\set{1,3,5,6}$, $\set{1,2,5,7}$, $\set{2,3,4,6}$, respectively.
    In this example, $D_{\mat_1, \mat_2}(\set{1,2,3})$ has a directed $s$-$t$ path $P = (s, 5, 2, 6, t)$ without shortcuts and
    $\set{1,3,5,6}$ is a common independent set of $\mat_1$ and $\mat_2$.}
    \label{fig:opt}
\end{figure}

Let $P$ be a directed path from $s$ to $t$ in $D(I)$ without shortcuts. 
By the definition of $D(I)$, $\size{V(P) \cap I}$ is one less than $\size{V(P) \setminus I}$.
Moreover, we can prove that $I \xor (V(P) \setminus \set{s, t})$ is a common independent set of $\mat_1$ and $\mat_2$~\cite{Lawler1975}, meaning that the common independent set $I \xor (V(P) \setminus\set{s, t})$ of $\mat_1$ and $\mat_2$ is strictly larger than $I$.
It is easy to see that $D(I)$ has a directed path from $s$ to $t$ without shortcuts if and only if $D(I)$ has a directed path from $s$ to $t$.
The following lemma summarizes the above discussion and also proves that the converse direction also holds.

\begin{lemma}[Corollary 3.2 in~\cite{Lawler1975}]\label{lem:opt}
    Let $I$ be a common independent set of two matroids $\mat_1$ and $\mat_2$ and $D(I)$ be the auxiliary directed graph.
    Then, $I$ is a maximum common independent set of $\mat_1$ and $\mat_2$ if and only if $D(I)$ has no directed $s$-$t$ path.
\end{lemma}

Such a path $P$ in $D(I)$ is called an \emph{augmenting path}.
\Cref{lem:opt} is helpful to design an algorithm for enumerating all large maximal common independent sets in two matroids.

\section{Enumeration of maximum common independent sets}\label{sec:maximum}
We first consider the problem of enumerating all \emph{maximum} common independent sets of two matroids, which is indeed a special case of \textsc{Large Maximal Common Independent Set Enumeration}, where $\tau = {\rm opt}$.\footnote{By ${\rm opt}$, we mean the maximum cardinality of a common independent set of $\mat_1$ and $\mat_2$.}
It is known that this problem can be solved in amortized polynomial time using the algorithm in \cite{FUKUDA1995231}.
However, they did not explicitly give a delay bound of this algorithm.
In order to show an explicit delay bound, we give a polynomial-delay algorithm for \textsc{Maximum Common Independent Set Enumeration}, using a simple flashlight search technique (also known as binary partition and backtracking)~\cite{Tarjan:1975,Birmele:SODA:2013}.

In this technique, the algorithm enumerates solutions with a simple branching algorithm.
Let $S$ be a set and let $\mathcal S \subseteq 2^S$ be the set of solutions.
We choose an element $e \in S$ and branch into two cases: the first branch enumerates all solutions that include $e$; the second branch enumerates all solutions that exclude $e$.
By recursively calling this branching algorithm for unchosen elements, the algorithm enumerates all solutions in $\mathcal S$ without duplicates.
However, to show an upper bound of the delay of this algorithm, we need to prune branches that do not generate any solutions.
To this end, it suffices to solve a certain decision problem, called an \emph{extension problem}.
To enumerate maximum common independent sets of matroids, we define \textsc{Maximum Common Independent Set Extension} as follows.

\begin{definition}
    Given two matroids $\mat_1 = (S, \mathcal I_1)$ and $\mat_2 = (S, \mathcal I_2)$, and two disjoint subsets $\mathit{In}, \mathit{Ex} \subseteq S$,
    \textsc{Maximum Common Independent Set Extension} asks to \emph{find} a maximum common independent set $I$ of $\mat_1$ and $\mat_2$
    that satisfies $\mathit{In} \subseteq I$ and $\mathit{Ex} \cap I = \emptyset$.
\end{definition}

In what follows, we call conditions $\mathit{In} \subseteq I$ and $\mathit{Ex} \cap I = \emptyset$ the \emph{inclusion condition} and \emph{exclusion condition}, respectively.
Note that for any matroid $\mat$, $\mat\setminus \mathit{Ex}$ and $M \cont \mathit{In}$ are matroids.
The following proposition is straightforward but essential for solving the extension problem.

\begin{proposition}
    Let $\mat_1 = (S, \mathcal I_1)$ and $\mat_2 = (S, \mathcal I_2)$ be two matroids, and $\mathit{In}$ and $\mathit{Ex}$ be disjoint subsets of $S$.
    Suppose that $\mathit{In}$ is a common independent set of $\mat_1$ and $\mat_2$.
    Let $\mat'_1 = (\mat_1 \cont \mathit{In}) \setminus \mathit{Ex}$ and $\mat'_2 = (\mat_2 \cont \mathit{In}) \setminus \mathit{Ex}$.
    Then, there is a maximum common independent set $I$ of $\mat_1$ and $\mat_2$ that satisfies $\mathit{In} \subseteq I$ and $\mathit{Ex} \cap I = \emptyset$ if and only if there is a common independent set $I'$ of $\mat'_1$ and $\mat'_2$ with the cardinality $\size{I} - \size{\mathit{In}}$.    
\end{proposition}

By the above proposition, we can solve \textsc{Maximum Common Independent Set Extension} in polynomial time by using a polynomial-time algorithm for finding a maximum common independent set of two matroids~\cite{Lawler1975}.
Note that by using oracles for $\mat_1$ and $\mat_2$, we can check whether a subset of $S$ is independent in $\mat'_1$ and in $\mat'_2$ in time $\order{n + Q}$ and space $\order{n + \hat{Q}}$.

Now, we design a simple flashlight search algorithm, which is sketched as follows.
Let $\mathcal S(\mathit{In}, \mathit{Ex})$ be the set of maximum common independent sets of $\mat_1$ and $\mat_2$ that satisfy both the inclusion and exclusion conditions.
Clearly, the set of all maximum common independent sets of $\mat_1$ and $\mat_2$ corresponds to $\mathcal S(\emptyset, \emptyset)$.
By solving the extension problem, we can determine whether $\mathcal S(\mathit{In}, \mathit{Ex})$ is empty or not in polynomial time.
Moreover, for an element $e \in S \setminus (\mathit{In}\cup\mathit{Ex})$,
$\set{\mathcal S(\mathit{In} \cup \set{e}, \mathit{Ex}), \mathcal S(\mathit{In}, \mathit{Ex} \cup \set{e})}$ is a partition of $\mathcal S(\mathit{In}, \mathit{Ex})$.
Thus, we can enumerate all maximum common independent sets in $\mathcal S(\mathit{In}, \mathit{Ex})$ by recursively enumerating $\mathcal S(\mathit{In} \cup \set{e}, \mathit{Ex})$ and
$\mathcal S(\mathit{In}, \mathit{Ex} \cup \set{e})$.
We give a pseudo-code of our algorithm in \Cref{algo:maximum}.
Finally, we consider the delay of this algorithm. 
Let $\mathcal T$ be a recursion tree defined by the execution of \Cref{algo:maximum}.
As we output a maximum common independent set of $\mat_1$ and $\mat_2$ at each leaf node in $\mathcal T$, the delay of the algorithm is upper bounded by the ``distance'' of two leaf nodes times the running time required to processing each node in $\mathcal T$.  
The distance between the root and a leaf node of $\mathcal T$ is at most $n$ and thus, the distance between two leaf nodes in $\mathcal{T}$ is upper bounded by linear in $n$. 
The time complexity of each node in $\mathcal T$ is in $\order{\poly(n + Q)}$. 
Hence, the delay of \Cref{algo:maximum} is upper bounded by a polynomial in $n + Q$. 
More specifically, by using an $\order{{\rm opt}^{3/2}nQ}$-time and $\order{n^2 + \hat Q}$-space algorithm for finding a maximum common independent set of two matroids~\cite{doi:10.1137/0215066}, the following theorem follows.

\begin{theorem}\label{thm:maximum}
    We can enumerate all maximum common independent sets of $\mat_1$ and $\mat_2$ in 
    $\order{{\rm opt}^{3/2}n^2Q}$ delay and $\order{n^2 + \hat Q}$ space. 
\end{theorem}

\DontPrintSemicolon
\begin{algorithm}[t]
    \caption{A polynomial-delay and polynomial-space algorithm for enumerating all maximum common independent sets in two matroids.}
    \label{algo:maximum}
    \Procedure{\Maximum{$\mat_1 = (S, \mathcal I_1), \mat_2 = (S, \mathcal I_2)$}}{
        \RecMaximum{$\mat_1, \mat_2, \emptyset, \emptyset$}\;
    }
    \Procedure{\RecMaximum{$\mat_1, \mat_2, \mathit{In}, \mathit{Ex}$}}{
        \lIf{$\mathit{In} \cup \mathit{Ex} = S$}{
            Output $\mathit{In}$, \Return
        }
        Choose an arbitrary $e \in S \setminus (\mathit{In}\cup\mathit{Ex})$\;
        \If{there is a maximum common independent set $I'$ of $\mat_1$ and $\mat_2$ that satisfies both $(\mathit{In} \cup \set{e}) \subseteq I'$ and $\mathit{Ex} \cap I' = \emptyset$}{
            \RecMaximum{$\mat_1, \mat_2, \mathit{In} \cup \set{e}, \mathit{Ex}$}
        }
        \If{there is a maximum common independent set $I'$ of $\mat_1$ and $\mat_2$ that satisfies both $\mathit{In} \subseteq I'$ and $(\mathit{Ex} \cup \set{e}) \cap I' = \emptyset$}{
            \RecMaximum{$\mat_1, \mat_2, \mathit{In}, \mathit{Ex} \cup \set{e}$}
        }        
    }
\end{algorithm}

\section{Enumeration of large maximal common independent sets}\label{sec:large}
We propose a polynomial-delay and polynomial-space algorithm for enumerating maximal common independent sets of two matroids $\mat_1$ and $\mat_2$ with the cardinality at least $\tau$, namely \textsc{Large Maximal Common Independent Set Enumeration}.
From \Cref{thm:maximum}, if $\tau = {\rm opt}$, we can enumerate all maximum common independent sets of $\mat_1$ and $\mat_2$ in polynomial delay and polynomial space. 
Thus, in this section, we assume that $\tau < {\rm opt}$.

Our proposed algorithm is based on \emph{reverse search}~\cite{Avis::1996}, which is one of the frequently used techniques to design efficient enumeration
algorithms~\cite{DBLP:conf/wads/KanteLMNU15,DBLP:conf/swat/MakinoU04,DBLP:journals/siamcomp/TsukiyamaIAS77,Kobayashi:WG:2022,DBLP:conf/mfcs/KuritaK20}.
One may expect that a flashlight search algorithm similar to that described in the previous section could be designed for \textsc{Large Maximal Common Independent Set Enumeration}, because finding a \emph{maximal} solution is usually easier than finding a \emph{maximum} solution. 
However, this intuitive phenomenon does not hold for extension problems.
In particular, the problem of finding a maximal matching in a bipartite graph that satisfies an exclusion condition is \NP-complete~\cite{Kobayashi:WG:2022,DBLP:conf/fct/CaselFGMS19}.
As the set of all matchings in a bipartite graph can be described as the set of common independent sets of matroids, the extension problem for \textsc{Large Maximal Common Independent Set Enumeration} is also \NP-complete.

Before delving into our algorithm, we briefly sketch an overview of the reverse search technique. 
Let $\mathcal S$ be the set of solutions.
In the reverse search technique, we define a set of ``special solutions'' $\mathcal R \subseteq \mathcal S$ and a rooted forest (i.e., a set of rooted trees) on $\mathcal S$ such that $\mathcal R$ is the set of its roots.
Suppose that we can enumerate $\mathcal R$ efficiently.
Then, we can enumerate all solutions in $\mathcal S$ by solely traversing the rooted forest from each root solution in $\mathcal R$.
To this end, for a non-root solution $X$ in $\mathcal S \setminus \mathcal R$, it suffices to define its parent $\parent{X}$ in an appropriate manner.
More specifically, to define a rooted forest on $\mathcal S$, this parent-child relation must have no cycles.
Moreover, to traverse this rooted forest, we need to efficiently enumerate the children of each internal node in the rooted forest.

Now, we turn back to our problem.
In the following, we may simply refer to maximal common independent sets of $\mat_1$ and $\mat_2$ with cardinality at least $\tau$ as \emph{solutions}.
We define the set of maximum common independent sets of $\mat_1$ and $\mat_2$ as the root solutions $\mathcal R$.
We can efficiently enumerate $\mathcal R$ by \Cref{thm:maximum}.
To define the parent of a solution not in $\mathcal R$, fix an arbitrary maximum common independent set $R$ of $\mat_1$ and $\mat_2$.
Let $I$ be a maximal common independent set of $\mat_1$ and $\mat_2$ with $\size{I} < \size{R}$.
We consider two matroids $\mat_1(R, I) \coloneqq (\mat_1 \rest(R \cup I))\cont (R \cap I)$ and 
$\mat_2(R, I) \coloneqq (\mat_2 \rest(R \cup I))\cont (R \cap I)$ as well as the auxiliary directed graph $D(R, I) \coloneqq D_{\mat_1(R, I), \mat_2(R, I)}(I\setminus R)$.
Let us note that the vertex set of $D(R, I)$ is $(R \xor I) \cup \set{s,t}$.
Since $R$ and $I$ are independent in both $\mat_1$ and $\mat_2$, by~\Cref{prop:union_intersection}, $R \setminus I$ and $I \setminus R$ are independent in $\mat_1(R, I)$ and $\mat_2(R, I)$ as well.
Moreover, as $|R| > |I|$, we have $|R \setminus I| > |I \setminus R|$.
Thus, by \Cref{lem:opt}, $D(R, I)$ has a directed $s$-$t$ path. 
Let $P$ be a directed $s$-$t$ path $(v_1 = s, v_2, \dots, v_{2k+1}=t)$ in $D(R, I)$ without shortcuts.
We first show that $I \xor \set{v_2, v_3}$ is a common independent set of $\mat_1$ and $\mat_2$. 

\begin{lemma}\label{lem:common_independent}
    Let $P = (v_1 = s, \dots, v_{2k+1} = t)$ be a directed $s$-$t$ path without shortcuts in $D(R, I)$.
    Then $P$ has at least four vertices and $I \xor \set{v_2, v_3}$ is also a common independent set of $\mat_1$ and $\mat_2$.
\end{lemma}
\begin{proof}
    We first show that $P$ has at least four vertices.
    As $s$ is not adjacent to $t$ in $D(R, I)$, $P$ has at least three vertices.
    If $P = (v_1 = s, v_2, v_3 = t)$, then $(I \setminus R) \cup \set{v_2}$ is a common independent set of $\mat_1(R, I)$ and $\mat_2(R, I)$.
    However, by~\Cref{prop:maximal}, $I \setminus R$ is a maximal common independent set of $\mat_1(R, I)$ and $\mat_2(R, I)$, a contradiction.

    We next show that $I \xor \set{v_2, v_3}$ is also a common independent set of $\mat_1$ and $\mat_2$.
    By the definition of $D(R, I)$, $v_2 \in R \setminus I$ and $v_3 \in I \setminus R$.
    This implies that $v_2 \notin I$ and $v_3 \in I$.
    Moreover, since $D(R, I)$ has arcs $(s, v_2)$ and $(v_2, v_3)$,
    $(I \setminus R) \cup \set{v_2}$ is independent in $\mat_1(R, I)$ and $(I \setminus R) \xor \set{v_2, v_3}$ is independent in $\mat_2(R, I)$.
    By~\Cref{prop:union_intersection}, $I \cup \set{v_2}$ and $I \xor \set{v_2, v_3}$ are independent in $\mat_1$ and in $\mat_2$, respectively.
    As $I \xor \set{v_2, v_3} = (I \setminus \set{v_3}) \cup \{v_2\}$ is a subset of $I \cup \set{v_2}$, $I\xor \set{v_2, v_3}$ is a common independent set of $\mat_1$ and $\mat_2$.
\end{proof}

We define the parent $\parent{I}$ of $I$ as follows.
To ensure the consistency of defining its parent, we choose a path $P$ from $s$ to $t$ without shortcuts in $D(R, I)$ in a certain way, and hence the path $P$ is determined solely by the pair $R$ and $I$.
We define the parent of $I$ (under $R$) as $\comp{I \xor \set{v_2, v_3}}$, where $\comp{X}$ is an arbitrary maximal common independent set of $\mat_1$ and $\mat_2$ containing $X$ for a common independent set $X$ of $\mat_1$ and $\mat_2$.
Similarly, we choose a maximal common independent set $\comp{X}$ in a certain way, and hence $\comp{X}$ is determined solely by $X$.
Therefore, by \Cref{lem:common_independent,lem:opt},
for a maximal common independent set $I$ of $\mat_1$ and $\mat_2$ with $\size{I} < \size{R}$, the parent of $I$ (under $R$) is uniquely determined and we denote it as $\parent{I}$.
In the following, we claim that this parent-child relation defines a rooted forest on the solutions whose roots belong to $\mathcal R$.
A key to this is a certain ``monotonicity'', which will be proven in \Cref{lem:card}: For any solution $I$ with $\size{I} < \size{R}$, it holds that $\size{R \xor I} > \size{R \xor \parent{I}}$.
Given these, from any solution $I$ with $|I| < |R|$, we can ``reach'' a maximum common independent set of $\mat_1$ and $\mat_2$ (not necessarily to be $R$) by iteratively taking its parent at most $n$ times as $\size{R \xor I} \le n$.
To show this monotonicity, we give the following technical lemma, whose proof is deferred to the end of this section.

\begin{restatable}{lem}{pone}\label{lem:+1}
    Let $I$ be a maximal common independent set of $\mat_1$ and $\mat_2$ with $\size{I} < \size{R}$ and $e \in I$ and $f \in S\setminus I$.
    If $D(I)$ has two arcs $(s, f)$ and $(f, e)$, then $I \xor \set{e, f}$ is a common independent set of $\mat_1$ and $\mat_2$.
    Moreover, $\size{\comp{I \xor \set{e, f}}} \le \size{I} + 1$.
\end{restatable}

Now, we prove the aforementioned ``monotonicity''.

\begin{lemma}\label{lem:card}
    Let $I$ be a maximal common independent set of $\mat_1$ and $\mat_2$ with $\size{I} < \size{R}$.
    Then, the following three properties on $\parent{I}$ are satisfied.
    \begin{itemize}
        \item $\size{I} \le \size{\parent{I}}$,
        \item $\size{R \xor I} > \size{R \xor \parent{I}}$, and
        \item $\size{I \xor \parent{I}} \le 3$.
    \end{itemize}
\end{lemma}
\begin{proof}
    As $\size{R \setminus I} > \size{I \setminus R}$, by~\Cref{lem:opt}, there is a directed path $P = (v_1, \dots, v_{2k+1})$ from $s = v_1$ to $t = v_{2k + 1}$ without shortcuts in $D(R, I)$.
    By~\Cref{lem:common_independent}, $I' = I \xor \set{v_2, v_3}$ is a common independent set of $\mat_1$ and $\mat_2$.
    The first property $\size{I} \le \size{\parent{I}}$ follows from
    \begin{align*}
        \size{\parent{I}} = \size{\comp{I \xor \set{v_2, v_3}}} \ge \size{I \xor \set{v_2, v_3}} = \size{I},
    \end{align*}
    as $v_2 \notin I$ and $v_3 \in I$.
    Since $v_2 \in R \setminus I$ and $v_3 \in I \setminus R$, we have $\size{R \xor I'} = \size{R \xor I} - 2$.
    If $D(I)$ has arcs $(s, v_2)$ and $(v_2, v_3)$, by~\Cref{lem:+1}, $\size{\comp{I'}} \le \size{I} + 1$, which yields that
    \begin{align*}
        \size{R \xor \parent{I}} = \size{R \xor \comp{I'}} \le \size{R \xor I'} + 1 = \size{R \xor I} - 1,
    \end{align*}
    where the inequality $\size{R \xor \comp{I'}} \le \size{R \xor I'} + 1$ follows from $\size{\comp{I'}} \le \size{I} + 1 = \size{I'} + 1$, meaning that $\comp{I'} \setminus I'$ contains at most one element.
    This also shows that
    \begin{align*}
        \size{I \xor \parent{I}} = \size{I \xor \comp{I'}} \le \size{I \xor I'} + 1 \le \size{I \xor (I \xor \set{v_2, v_3})} + 1 = 3.
    \end{align*}
    
    Thus, it suffices to show that $D(I)$ has these arcs $(v_2, v_3)$ and $(s, v_2)$.
    Let $\mat_1(R, I) = (\mat_1 \rest (R \cup I)) \cont (R \cap I)$ and $\mat_2(R, I) = (\mat_2 \rest (R \cup I)) \cont (R \cap I)$.
    As $D(R, I)$ has the arc $(v_2, v_3)$, $(I \setminus R) \cup \set{v_2}$ and $((I \setminus R) \cup \set{v_2}) \setminus \set{v_3}$ are dependent and independent in $\mat_2(R, I)$, respectively.
    By~\Cref{prop:union_intersection}, 
    \begin{align*}
        &((I\setminus R) \cup \set{v_2}) \cup (R \cap I) = I \cup \set{v_2} \text{ and}\\
        &(((I \setminus R) \cup \set{v_2}) \setminus \set{v_3}) \cup (R \cap I) = (I \cup \set{v_2}) \setminus \set{v_3}
    \end{align*}
    are dependent and independent in $\mat_2$, respectively.
    This implies that $D(I)$ has an arc $(v_2, v_3)$.
    A similar argument for $(s, v_2)$ and $\mat_1(R, I)$ proves that $D(I)$ has arc $(s, v_2)$, completing the proof of this lemma.
\end{proof}

Now, we are ready to describe our algorithm, which is also shown in \Cref{algo:maximal}.
We assume that the size of a maximum common independent set of $\mat_1$ and $\mat_2$ is at least $\tau$ as otherwise we do nothing.
We first enumerate the set $\mathcal R$ all maximum common independent sets of $\mat_1$ and $\mat_2$.
This can be done in polynomial delay and polynomial space using the algorithm in \Cref{thm:maximum}.
We choose an arbitrary $R \in \mathcal R$ and for each $I \in \mathcal R$, we enumerate all solutions that belong to the component containing $I$ in the rooted forest defined by the parent-child relation.
This is done by calling \RecMaximal{$\mat_1, \mat_2, I, R, \tau$}.
The procedure \RecMaximal{$\mat_1, \mat_2, I, R, \tau$} recursively generates solutions $I'$ with $I = \parent{I'}$.
We would like to emphasize that the algorithm only generates solutions $I'$ with $\size{I'} \ge \tau$.

We first claim that all the solutions are generated by this algorithm.
To see this, consider an arbitrary solution $I$.
Define a value $v(I)$ as
\begin{align*}
    v(I) = \begin{dcases}
        0 & \text{if } \size{I} = \size{R}\\
        \size{I \xor R} & \text{otherwise}.
    \end{dcases}
\end{align*}
We prove the claim by induction on $v(I)$.
Suppose that $v(I) = 0$.
In this case, $I$ is a maximum common independent set of $\mat_1$ and $\mat_2$, as $\size{I \xor R} = 0$ if and only if $I = R$.
Then, $I$ is obviously generated as we call \RecMaximal{$\mat_1, \mat_2, I, R, \tau$} for all $I \in \mathcal R$.
Suppose that $v(I) > 0$.
Then, $I$ is a maximal common independent set of $\mat_1$ and $\mat_2$ with $\size{I} < \size{R}$.
We assume that all the solutions $I'$ with $v(I') = \size{R \xor I'} < \size{R \xor I}$ is generated by the algorithm.
By~\Cref{lem:card}, we have $\size{\parent{I}} \ge \size{I} \ge \tau$ and $\size{R \xor \parent{I}} < \size{R \xor I}$, which implies that $\parent{I}$ is generated by the algorithm.
By the definition of parent, we have $\parent{I} = \comp{I \xor \set{u, v}}$ and $\size{\parent{I} \xor I} \ge 2$ for some $u, v \in S$.
Moreover, by~\Cref{lem:card}, $\size{I \xor \parent{I}} \le 3$, the child $I$ of $\parent{I}$ is computed at line~\ref{line:comp:child}.
Thus, $I$ is generated by the algorithm as well.

We next claim that all the solutions are generated without duplication.
Since we only call \RecMaximal{$\mat_1, \mat_2, I', R, \tau$} for $I = \parent{I'}$, it holds that $v(I) < v(I')$.
As $v(I) \le n$ for every solution $I$, by the uniqueness of the parent of non-maximum solutions, the algorithm generates each solution exactly once. 
This concludes that the algorithm correctly enumerates all maximal common independent sets of $\mat_1$ and $\mat_2$ of cardinality at least $\tau$.

Finally, we discuss the running time of the algorithm.
We first enumerate all maximum common independent sets of $\mat_1$ and $\mat_2$.
This can be done in time $O(n^{7/2}Q)$ delay.
For each solution $I$, we compute $\parent{I}$ as follows.
We construct the graph $D(R, I)$ that has $n + 2$ nodes and $O(n^2)$ arcs.
This can be done by using $O(n^2Q)$ queries to the oracles for $\mat_1$ and $\mat_2$.
To find the path $P$ from $s$ to $t$ without shortcuts, we just compute a shortest path from $s$ to $t$, which can be done in $O(n^2)$ time.
Thus, we can compute $I \xor \{v_2, v_3\}$ in $O(n^2Q)$ time as well.
From $I \xor \{v_2, v_3\}$, $\comp{I \xor \{v_2, v_3\}}$ can be computed in $O(nQ)$ time.
Thus, we can compute $\parent{I}$ from $I$ in time $O(n^2Q)$.

For each call \RecMaximal{$\mat_1, \mat_2, I, R, \tau$}, we output exactly one solution.
Moreover, the running time of computing all children of $I$ is $O(n^5Q)$.
This can be seen as there are $O(n^3)$ candidates $I'$ of children and we can check in $O(n^2Q)$ time whether a candidate $I'$ is in fact a child of $I$.
Thus the delay of the algorithm is upper bounded by the time elapsed between two consecutive calls.
As the depth of the rooted forest defined by recursive calls is at most $n$, this can be upper bounded by $O(n^6Q)$. 
As for the space complexity, by~\Cref{thm:maximum}, we can enumerate all maximum common independent sets of $\mat_1$ and $\mat_2$ in $O(n^2 + \hat{Q})$ space.
In \RecMaximal, we need to store local variables $I$ and $X$ in each recursive call, which can be done in space $O(n)$.
As the depth of the rooted forest is at most $n$, the space usage for local variables is $\order{n^2}$ in total.
For each candidate $I'$, we can check in $\order{n^2 + \hat{Q}}$ whether $I'$ is a maximal common independent set of $\mat_1$ and $\mat_2$ and whether $I = \parent{I'}$.
Overall, we have the following theorem.

\begin{theorem}\label{thm:large}
    There is an $\order{n^6Q}$-delay and $\order{n^2 + \hat Q}$-space algorithm for enumerating maximal common independent sets in two matroids with the cardinality at least $\tau$.
\end{theorem}

\begin{algorithm}[t]
    \SetAlgoLined
    \Procedure{\Maximal{$\mat_1 = (S, \mathcal I_1), \mat_2 = (S, \mathcal I_2), \tau$}}{
        Choose arbitrary $R \in \mathcal R$\;
        \lForEach(\tcp*[f]{Use \Cref{algo:maximum}.}){$I \in \mathcal R$}{
            \RecMaximal{$\mat_1, \mat_2, I, R, \tau$}
        }
    }
    \Procedure{\RecMaximal{$\mat_1, \mat_2, I, R, \tau$}}{
        Output $I$\;
        
        \ForEach{$X \in \binom{S}{3} \cup \binom{S}{2}$}{\label{line:comp:X}
            $I' \gets I \xor X$\;\label{line:comp:child}
            \If{$I'$ is a maximal common independent set of $\mat_1$ and $\mat_2$ such that $\tau \le \size{I'} < \size{R}$ and $I = \parent{I'}$}{
                \RecMaximal{$\mat_1, \mat_2, I', R, \tau$}\;
            }
        }
    }
    \caption{A polynomial-delay and polynomial-space algorithm for enumerating all maximal common independent sets of $\mat_1$ and $\mat_2$ with the cardinality at least $\tau$.}
    \label{algo:maximal}
\end{algorithm}


\subsection{Proof of Lemma~\ref{lem:+1}}\label{subsec:lem}
To complete our proof of \Cref{thm:large}, we need to show the correctness of \Cref{lem:+1}.
To this end, we focus on $D(I)$.
Since $I$ is a maximal common independent set of $\mat_1$ and $\mat_2$ with $\size{I} < \size{R}$, 
$D(I)$ has a directed $s$-$t$ path $P = (v_1 = s, v_2 = f, v_3 = e, \dots, v_{2k+1} = t)$.
By the definition of $D(I)$, $I \xor \set{e, f}$ is a common independent set of $\mat_1$ and $\mat_2$.
We first show that $I \xor \set{e, f}$ becomes dependent in $\mat_1$ when we add an element $f'$ in $N^+_{D(I)}(e)$.

\begin{lemma}\label{lem:-}
    Let $I$ be a maximal common independent set of $\mat_1$ and $\mat_2$,
    $e$ be an element in $I$, and
    $f_1$ and $f_2$ be distinct two elements in $N^+_{D(I)}(e)$.
    Then, $I' \coloneqq (I \setminus \set{e}) \cup \set{f_1, f_2}$ is dependent in $\mat_1$.
\end{lemma}
\begin{proof}
    Since $D(I)$ has arcs $(e, f_1)$ and $(e, f_2)$, 
    both $(I\setminus \set{e}) \cup \set{f_1}$ and $(I\setminus \set{e}) \cup \set{f_2}$ are independent in $\mat_1$, and 
    $I \cup \set{f_1}$ and $I \cup \set{f_2}$ are dependent in $\mat_1$.
    Thus, $\mat_1$ has two circuits $C_1$ and $C_2$ that contain $\set{e, f_1}$ and $\set{e, f_2}$, respectively.
    By the circuit elimination axiom, there is a circuit $C_3 \subseteq (C_1 \cup C_2) \setminus \set{e}$.
    Since $I'$ contains $(C_1 \cup C_2) \setminus \set{e}$, 
    it also contains $C_3$ and hence is dependent in $\mat_1$.
\end{proof}

\begin{lemma}\label{lem:+}
    Let $I$ be a maximal common independent set of $\mat_1$ and $\mat_2$,
    $e$ be an element in $I$, and
    $f_1$ and $f_2$ be distinct two elements in $N^-_{D(I)}(e)$.
    Then, $I' \coloneqq (I \setminus \set{e}) \cup \set{f_1, f_2}$ is dependent in $\mat_2$
\end{lemma}
\begin{proof}
    Since $D(I)$ has arcs $(f_1, e)$ and $(f_2, e)$, 
    both $I\xor \set{e, f_1}$ and $I\xor \set{e, f_2}$ are independent in $\mat_2$ and 
    $I \cup \set{f_1}$ and $I \cup \set{f_2}$ are dependent in $\mat_2$.
    Thus, $\mat_2$ has two circuits $C_1$ and $C_2$ that contain $\set{e, f_1}$ and $\set{e, f_2}$, respectively.
    By the circuit elimination axiom, there is a circuit $C_3 \subseteq (C_1 \cup C_2) \setminus \set{e}$.
    Since $I'$ contains $(C_1 \cup C_2) \setminus \set{e}$, 
    it also contains $C_3$ and hence is dependent in $\mat_2$.
\end{proof}

We show that $\comp{I \xor \set{e, f}}$ does not contain any element in $S \setminus (I \cup N^+_{D(I)}(e) \cup N^-_{D(I)}(e))$.

\begin{lemma}\label{lem:non}
    Let $I$ be a maximal common independent set of $\mat_1$ and $\mat_2$,
    $e$ be an element in $I$, and
    $f$ be an element in $S \setminus (I \cup N^+_{D(I)}(e) \cup N^-_{D(I)}(e))$.
    Then, $I \xor \set{e, f}$ is dependent in at least one of $\mat_1$ or $\mat_2$.
\end{lemma}
\begin{proof}
    From the maximality of $I$, $I \cup \set{f}$ is dependent in at least one of $\mat_1$ or $\mat_2$.
    Suppose that $I \cup \set{f}$ is dependent on $\mat_2$.
    Then, $I \cup \set{f}$ contains at least one circuit $C$ of $\mat_2$ containing $f$.
    We show that $I \cup \set{f}$ contains only one circuit of $\mat_2$.
    If $I \cup \set{f}$ contains another circuit $C'$ with $f \in C'$, by the circuit elimination axiom, $(C \cup C') \setminus \set{f}$ contains a circuit, which contradicts the fact that $I$ is independent in $\mat_2$.
    Thus, $\mat_2$ has the unique circuit $C$, which is contained in $I \cup \set{f}$.
    Observe that $N^+_{D(I)}(f) \cap I = C \setminus \set{f}$, since $(I \cup \set{f}) \setminus \set{e'}$ is independent in $\mat_2$ for $e' \in C$ due to the minimality of $C$.
    As $f \in S \setminus (I \cup N^+_{D(I)}(e) \cup N^-_{D(I)}(e))$, we have $e \notin C$.
    Hence, $(I\setminus \set{e})\cup\set{f}$ contains $C$, that is, $(I\setminus \set{e})\cup\set{f}$ is dependent in $\mat_2$.
    When $I \cup \set{f}$ is dependent in $\mat_1$, $(I\setminus\set{e})\cup\set{f}$ is also dependent in $\mat_1$ from a similar discussion.    
\end{proof}

Now we are ready to prove \Cref{lem:+1}.

\pone*

\begin{proof}
    By the definition of $D(I)$, $I\xor \set{e, f}$ is a common independent set since $D(I)$ has two arcs $(s, f)$ and $(f, e)$.
    Thus, $\comp{I \xor \set{e, f}}$ is a maximal common independent set of $\mat_1$ and $\mat_2$.
    We show that $\size{\mu(I \xor \set{e, f})} \le \size{I} + 1$.
    Since $f$ is contained in $N^-_{D(I)}(e)$, $\mu(I \xor \set{e, f})$ does not contain elements in $N^-_{D(I)}(e)$ except for $f$ by \Cref{lem:+}.
    Moreover, by \Cref{lem:-}, $\mu(I \xor \set{e, f})$ contains at most one element in $N^+_{D(I)}(e)$.
    Finally, $\mu(I \xor \set{e, f})$ does not contain any element in $S \setminus (I \cup N^+_{D(I)}(e) \cup N^-_{D(I)}(e))$ by \Cref{lem:non}.
    Therefore, $\mu(I \xor \set{e, f}) \setminus (I \xor \set{e, f})$ contains at most one element.
    Since $\size{I \xor \set{e, f}} = \size{I}$, $\size{\mu(I \xor \set{e, f})} \le \size{I} + 1$.
\end{proof}


\section{Beyond common independent sets in two matroids: matroid matching}\label{sec:parity}
In the previous section, we give a polynomial-delay and polynomial-space algorithm for enumerating maximal common independent sets of two matroids with cardinality at least $\tau$.
To show the correctness of our algorithm,
the following properties are essential.

\begin{enumerate}
    \item All maximum solutions $\mathcal R$ (i.e., maximum common independent sets of $\mat_1$ and $\mat_2$) can be enumerated in polynomial delay and polynomial space.
    \item We define a rooted forest on the set of all solutions $\mathcal S$ via $\parent{\cdot}$ such that $\mathcal R$ is the set of its roots.
    \item For a maximal feasible solution $M$ that is not in $\mathcal R$, we have $|M| \le |\parent{M}|$.
    \item We can enumerate the children of each solution in $\mathcal S$ in polynomial time.
\end{enumerate}

Similarly to the algorithm of \Cref{thm:large}, by the properties 1, 2, and 4, we can enumerate all solutions using the reverse search technique. 
Moreover, the property 3 ensures the ``monotonicity'' of solutions on root-leaf paths defined by $\parent{\cdot}$.
Thus, it suffices to define the parent-child relationship $\parent{\cdot}$ that satisfies the above four properties for maximal matroid matchings.
In what follows, let $\mat = (S, \mathcal I)$ be a matroid and let $G = (S, E)$ be a graph.
Moreover, we assume that $(\mat, G)$ is tractable.
For $F \subseteq E$, we denote $\bigcup_{e \in F}e$ as $V(F)$.

\subsection{Enumerating maximum matroid matchings}
We first give a polynomial-delay algorithm for enumerating maximum matchings, proving the property 1.
Our algorithm runs in $\poly(n + Q)$ delay and $\poly(n + \hat Q)$ space.
The outline of this algorithm is identical to that described in \Cref{sec:maximum}.
To enumerate all maximum matching, we consider the following extension problem defined.

\begin{definition}
    Given a tractable pair $(\mat, G)$ and two disjoint subsets of edges $\mathit{In}, \mathit{Ex} \subseteq E$, \textsc{Maximum Matroid Matching Extension} asks to find a maximum matching $M \subseteq E$ of $(\mat, G)$ such that $\mathit{In} \subseteq M$ and $M \cap \mathit{Ex} = \emptyset$.
\end{definition}

We show that 
this problem can be solved in polynomial time, provided that $(\mat, G)$ is tractable.
By the tractability of $(\mat, G)$,
we can find a maximum matching of $(\mat', G')$ in polynomial time, where $\mat' = (\mat \cont V(\mathit{In})) \rest V(\mathit{Ex})$ and $G' = G[S \setminus V(\mathit{In})] - \mathit{Ex}$.
The following corollary immediately give a polynomial-time algorithm for \textsc{Maximum Matroid Matching Extension} of tractable pairs.

\begin{corollary}
    Let $\mathbf M = (S, \mathcal I)$ be a matroid, 
    $G$ be a graph $(S, E)$, and
    $\mathit{In}$ and $\mathit{Ex}$ be a disjoint set of edges of $G$.
    Let $\mat' = (\mat \cont V(\mathit{In})) \rest V(\mathit{Ex})$ and $G' = G[S \setminus V(\mathit{In})] - \mathit{Ex}$.
    Then, there is a maximum matching $M$ of $(\mat, G)$ that satisfies $\mathit{In} \subseteq M$ and $M \cap \mathit{Ex} = \emptyset$ 
    if and only if $(\mat', G')$ has a matching with the cardinality $\size{M} - \size{\mathit{In}}$.
\end{corollary}

\begin{theorem}\label{thm:parity:maximum}
    For a tractable pair $(\mat, G)$,
    we can enumerate all maximum matchings with $\poly(n + Q)$ delay and $\poly(n + \hat Q)$ space.
\end{theorem}

\subsection{Enumerating of large maximal matroid matchings}
In this subsection, we give a polynomial-delay and polynomial-space enumeration algorithm for
\textsc{Large Maximal Matroid Matching Enumeration} for tractable pairs $(\mat, G)$.
The outline of our algorithm is again identical to that in \Cref{sec:large}.
In the following, we show analogous lemmas of~\Cref{lem:+1} and \ref{lem:card}.
Fix a maximum matching $R$ of $(\mat, G)$.
For a matching $M$ of $(\mat, G)$, we denote by $\mu(M)$ an arbitrary maximal matching of $(\mat, G)$ that contains $M$.

\begin{lemma}\label{lem:ana:common_independent}
    Let $M$ be a maximal matching of $(\mat, G)$ with $|M| < |R|$.
    Then, there is a pair of edges $e \in M \setminus R$ and $f \in R \setminus M$
    such that $(M \setminus \set{e}) \cup\set{f}$ is a matching of $(\mat, G)$.
    Moreover, $\size{\mu((M \setminus \set{e}) \cup \set{f})} \le \size{M} + 2$.
\end{lemma}
\begin{proof}
    We first show that there is a pair of edges $e \in M \setminus R$ and $f \in R \setminus M$
    such that $(M \setminus \set{e}) \cup\set{f}$ is a matching of $(\mat, G)$.
    Since $V(R)$ and $V(M)$ are independent in $\mat$ and $\size{V(R)} > \size{V(M)}$, $V(R) \setminus V(M)$ has an element $x$ such that $V(M) \cup \set{x}$ is independent in $\mat$.
    Moreover, since $R$ is a matchings of $G$, $R \setminus M$ contains an edge $f = \set{x, x'}$.
    
    Suppose that $M \cup \set{f}$ is not a matching of $G$, 
    Then, $M$ has an edge $e = \set{x', x''}$ with $x'' \neq x$ as $V(M)$ does not contain $x$.
    In this case, $M' = (M \setminus \set{e}) \cup \set{f}$ is a matching of $G$.
    Moreover,
    \begin{align*}
        V(M') = (V(M) \setminus \set{x', x''}) \cup \set{x, x'} = (V(M) \setminus \set{x''})\cup \set{x}
    \end{align*}
    is independent in $\mat$, which implies that $M'$ is a matching of $(\mat, G)$.

    Suppose otherwise that $M \cup \set{f}$ is a matching of $G$.
    As $M$ is a maximal matching of $(\mat, G)$, $V(M \cup \set{f})$ is dependent in $\mat$.
    Since $V(M) \cup \{x\}$ is independent in $\mat$, 
    $V(M \cup \set{f})$ has a circuit $C$ that contains $x'$.
    Notice that $\set{x, x'}$ is independent in $\mat$ as $\{x, x'\} \subseteq V(R)$.
    This implies that $C \setminus \set{x, x'}$ is non-empty.
    Thus, $M$ contains an edge $e$ that has an element of $C \setminus \set{x, x'}$, which implies that $V((M \setminus\set{e}) \cup \set{f})$ is independent in $\mat$.
    Therefore, $M' = (M \setminus \set{e}) \cup \set{f}$ is a matching of $(\mat, G)$.

    We next show that $\size{\mu(M')} \le \size{M} + 2$.
    In what follows, we let $e = \set{y, z}$.
    Suppose that $\mu(M') \setminus M$ has three distinct edges $e_1 = \set{x_1, x'_1}$, $e_2 = \set{x_2, x'_2}$, and $e_3 = \set{x_3, x'_3}$.
    Since $M$ is a maximal matching of $(\mat, G)$, for each $e_i$, $M$ has an edge that is incident to at least one of $x_i$ and $x'_i$ or $V(M \cup \set{e_i})$ is dependent in $\mat$.

    Suppose that there are two edges, say $e_1$ and $e_2$, in $\set{e_1, e_2, e_3}$ that are incident to $y$ or $z$.
    This implies that $\set{y, z} \subseteq V(\set{e_1, e_2, e_3})$, and hence we have
    \begin{align*}
        V(M \cup \set{e_3}) \subseteq (V(M) \setminus \set{y, z}) \cup V(\set{e_1, e_2, e_3}) \subseteq V(\mu(M')).
    \end{align*}
    As $M$ is a maximal matching of $(\mat, G)$, either $M \cup \set{e_3}$ is not a matching of $G$ or $V(M \cup \set{e_3})$ is dependent in $\mat$.
    If $M \cup \{e_3\}$ is not a matching of $G$, $(M \setminus \set{e}) \cup \set{e_1, e_2, e_3}$ is not a matching of $G$.
    Otherwise, $V(M \cup \set{e_3})$ is dependent in $\mat$, which contradicts the fact that $V(\mu(M'))$ is independent in $\mat$.
    
    Suppose next that $\set{e_1, e_2, e_3}$ has exactly one edge incident to $y$ or $z$.
    Without loss of generality, we assume that $e_1$ is incident to $y$.
    If $V(M \cup \set{e_2})$ is independent in $\mat$, by the maximality of $M$, $M \cup \set{e_2}$ is not a matching of $G$.
    However, as $(M \setminus \set{e}) \cup \set{e_1, e_2, e_3} \subseteq \mu(M')$, $M'$ is not a matching of $G$, a contradiction.
    Thus, $V(M \cup \set{e_2})$ is dependent in $\mat$.
    By an analogous argument, $V(M \cup \set{e_3})$ is also dependent in $\mat$.
    As $(V(M) \setminus \set{y, z}) \cup V(\set{e_1, e_2, e_3})$ is independent in $\mat$, 
    $V(M \cup \set{e_i})$ has a circuit $C_i$ with $z \in C_i$ for $i = 2,3$.
    By the circuit elimination axiom, we obtain a circuit $C' \subseteq (C_2 \cup C_3) \setminus \set{z}$.
    However, we have $C' \subseteq (V(M) \setminus \set{z}) \cup V(\set{e_2, e_3}) \subseteq V(\mu(M'))$, which contradicts the fact that $V(\mu(M'))$ is independent in $\mat$.
    
    Finally, suppose that $\set{e_1, e_2, e_3}$ has no edges incident to $y$ or $z$.
    In this case, for $i = 1,2,3$, $V(M \cup \set{e_i})$ is dependent in $\mat$, meaning that it has a circuit $C_i$.
    Since $V(\mu(M'))$ is independent in $\mat$, each $C_i$ contains at least one of $y$ and $z$.
    We classify circuits into two types: type-1 circuits contain both $y$ and $z$; type-2 circuits contain exactly one of $y$ and $z$.
    In the following, we distinguish several cases, each of which derives a contradiction, proving eventually that $|\mu(M')| \le |M| + 2$.

    We first consider the case where there are type-2 circuits, say $C_1$ and $C_2$, that contain $y$.
    As these circuits are type-2, neither $C_1$ nor $C_2$ contain $z$.
    By the circuit elimination axiom, there is a circuit $C_{1, 2}$ in $\mat$ with $C_{1,2} \subseteq (C_1 \cup C_2) \setminus \set{y}$.
    Since 
    \begin{align*}
        (C_1 \cup C_2) \setminus \set{y} &\subseteq (V(M \cup \set{e_1}) \cup V(M \cup \set{e_2})) \setminus \set{y}\\
        &\subseteq (V(M) \setminus \set{y, z}) \cup V(\set{e_1, e_2})\\
        &\subseteq V(\mu(M')),
    \end{align*}
    $V(\mu(M'))$ contains $C_{1, 2}$, contradicting the fact that $V(\mu(M'))$ is independent in $\mat$.
    Thus, there is no more than one type-2 circuit that contains $y$.
    Similarly, there is no more than one type-2 circuit that contains $z$.

    We consider the case $C_1$ and $C_2$ are type-2 such that $y \in C_1$ and $z \in C_2$.
    Without loss of generality, $C_3$ contains $y$.
    Since $C_1$ and $C_3$ are distinct, by the circuit elimination axiom, we obtain a circuit $C_{1, 3} \subseteq (C_1 \cup C_3) \setminus \set{y}$.
    Since both $V(\mu(M'))$ and $V(M)$ are independent in $\mat$, $z \in C_{1, 3}$ and $C_{1, 3} \cap \set{x_1, x'_1, x_3, x'_3} \neq \emptyset$.
    Thus, $C_{1, 3}$ is distinct from $C_2$, as $C_2 \cap \set{x_1, x'_1, x_3, x'_3} = \emptyset$, and hence we obtain a circuit $\hat C \subseteq (C_2 \cup C_{1, 3}) \setminus \set{z}$.
    Since
    \begin{align*}
        (C_{2} \cup C_{1,3}) \setminus \set{z}
        &\subseteq (V(M \cup \set{e_2}) \cup (V(M \cup \set{e_1, e_3}) \setminus \set{y}) \setminus \set{z}\\
        &\subseteq (V(M) \setminus \set{y, z}) \cup V(\set{e_1, e_2, e_3})\\
        &\subseteq V(\mu(M')),
    \end{align*}
    $V(\mu(M'))$ contains $\hat C$, a contradiction.
    
    We next consider the case where there are exactly two type-1 circuits, say $C_1$ and $C_2$.
    Without loss of generality, we assume that $y \in C_3$ and $z \notin C_3$. 
    By the circuit elimination axiom, there is a circuit $C_{1, 2} \subseteq (C_1 \cup C_2) \setminus \set{z}$.
    If $y \notin C_{1, 2}$, we have
    \begin{align*}
        C_{1, 2} \subseteq V(M \setminus \set{y, z}) \cup V(\set{e_1, e_2}) \subseteq V(\mu(M')),
    \end{align*}
    yielding a contradiction.
    Thus, suppose otherwise (i.e., $y \in C_{1,2}$).
    Since $C_3 \cap e_3 \neq \emptyset$ and $C_{1, 2} \cap e_3 = \emptyset$, we have $C_3 \neq C_{1, 2}$.
    Hence, we obtain a circuit $\hat C \subseteq (C_3 \cup C_{1, 2}) \setminus\set{y}$, which satisfies
    \begin{align*}
        \hat C \subseteq V(M \setminus \set{y, z}) \cup V(\set{e_1,e_2, e_3} \subseteq \mu(V(M')),
    \end{align*}
    yielding a contradiction.
    
    We finally consider the case where all $C_1$, $C_2$, and $C_3$ are type-1.
    By the circuit elimination axiom, we obtain circuits
    $C_{1, 2} \subseteq (C_1 \cup C_2) \setminus \set{y}$,
    $C_{2, 3} \subseteq (C_2 \cup C_3) \setminus \set{y}$, and
    $C_{3, 1} \subseteq (C_3 \cup C_1) \setminus \set{y}$.
    Note that $C_{i, j} \cap e_k = \emptyset$ for distinct $i, j, k$.
    Suppose that all of these circuits contain $z$.
    If $C_{1, 2} = C_{2, 3} = C_{3, 1}$,
    we have $C_{1, 2} \subseteq V(M)$, which contradicts the fact that $M$ is independent in $\mat$.
    Thus, we assume that $C_{1, 2} \neq C_{2, 3}$.
    By the circuit elimination axiom, we obtain a circuit $\hat C \subseteq (C_{1, 2} \cup C_{2, 3}) \setminus \set{z}$, which satisfies $\hat C \subseteq \mu(V(M'))$ as above.
    This leads to a contradiction as well.

    In all cases, we have a contradiction, which completes the proof.
\end{proof}

From the above lemma, for each maximal matching $M$ of $(\mat, G)$ with $|M| < |R|$,
there are two edges $e \in M \setminus R$ and $f \in R \setminus M$ such that $((M \setminus \set{e}) \cup \set{f})$ is also a matching of $(\mat, G)$.
We define the parent of $M$ as $\parent{M} = \mu((M \setminus \set{e}) \cup \set{f})$.
As in the same discussion in \Cref{sec:large}, $\parent{M}$ satisfies the following three conditions.

\begin{lemma}\label{lem:props}
    Let $M$ be a maximal matching of $(\mat, G)$ with $\size{M} < \size{\parent{M}}$.
    Then, the following three statements on $\parent{M}$ are satisfied.
    \begin{itemize}
        \item $\size{M} \le \size{\parent{M}}$,
        \item either $\size{R \xor M} > \size{R \xor \parent{M}}$ or $\size{M} < \size{\parent{M}}$, and
        \item $\size{M \xor \parent{M}} \le 4$.
    \end{itemize}
\end{lemma}
\begin{proof}
    Let $e \in M \setminus R$ and $f \in R \setminus M$ be the edges used in defining $\parent{M}$.
    Since
    \[
       \size{\parent{M}} = \size{\mu((M \setminus \set{e}) \cup \set{f}) \ge \size{M}},
    \]
    the first statement holds.
    
    To see the second statement, we consider the case $|\parent{M}| = |M|$, that is, $\mu((M \setminus \set{e}) \cup \set{f}) = (M \setminus \set{e}) \cup \set{f}$.
    Since $e \in M \setminus R$ and $f \in R \setminus M$,
    $\size{R \xor \parent{M}} = \size{R \xor M} - 2$, implying the second property.
    
    The third statement follows from, by letting $M' = (M \setminus \set{e}) \cup \set{f}$,
    \begin{align*}
        \size{M \xor \parent{M}} &= \size{M \setminus \mu(M')} + 
        \size{\mu(M') \setminus M}\\
        &\le \size{M \setminus M'} + \size{M' \setminus M} + \size{\mu(M') \setminus M'}\\
        &\le 4,
    \end{align*}
    where the last inequality follows from the proof of \Cref{lem:ana:common_independent}.
\end{proof}

The first statement of \Cref{lem:props} corresponds to the property 3.
Now, we are ready to prove the properties 2 and 4.
We first show property 2, that is, the graph defined by $\parent{\cdot}$ is a rooted forest in which each root solution is a maximum matching of $(\mat, G)$.
To this end, we assign an ordered pair of non-negative integers $v(M)$ to each maximal matching of $(\mat, G)$ as:
\begin{align*}
    v(M) = \begin{dcases}
        (0, 0) & \text{if } \size{I} = \size{R}\\
        (\size{R} - \size{M}, \size{M \xor R}) & \text{otherwise}.
    \end{dcases}
\end{align*}

For pairs of integers $(a, b)$ and $(c, d)$,
we define a lexicographic order $\preceq$, that is $(a, b) \preceq (c, d)$ if and only if $a < c$; or $a = c$ and $b < d$.
Moreover, we particularly write $(a, b) \prec (c, d)$ when $(a, b) \neq (c, d)$.
In our proof, we show that $v(\parent{M}) \prec v(M)$ for any maximal matching $M$ of $(\mat, G)$ with $|M| < |R|$, which implies the property $2$.

\begin{lemma}
    Let $M$ be a maximal matching of $(\mat, G)$. Then, $v(\parent{M}) \prec v(M)$.
\end{lemma}
\begin{proof}
    By \Cref{lem:props}, either $\size{R \xor M} > \size{R \xor \parent M}$ or $\size{M} < \size{\parent M}$ holds.
    When $\size{M} < \size{\parent{M}}$, we have $v(\parent{M}) \prec v(M)$ as $|R| - |\parent{M}| < |R| - |M|$.
    Otherwise, we have $\size{R \xor M} > \size{R \xor \parent M}$.
    In this case, again we have $v(\parent{M}) \prec v(M)$ as $|R \xor \parent{M}| < |R \xor M|$.
\end{proof}

Finally, we show property 4, that is, we can enumerate all children of $M$ in polynomial time.
By \Cref{lem:props}, for every child $M'$ of $M$, $|M \xor M'| \le 4$.
Thus, by guessing $M \xor \parent{M}$, we can generate all the children of $M$ in polynomial time.

\begin{theorem}\label{thm:parity:large}
    For a tractable pair $(\mat, G)$,
    \textsc{Large Maximal Matroid Matching Enumeration} can be solved in $\poly(n + Q)$ delay and $\poly(n + \hat Q)$ space.
\end{theorem}

\section{A framework to convert an enumeration algorithm with cardinality constraints into a ranked enumeration algorithm}
In this section, we give a framework to convert 
a polynomial-delay and polynomial-space enumeration algorithm with cardinality constraints into
a ranked enumeration algorithm that runs in linear incremental-delay and polynomial space.
Recall that an enumeration algorithm is called a \emph{ranked enumeration algorithm} if the algorithm enumerates solutions in a non-increasing order of their cardinality.

In what follows, we consider the following abstract problem.
Let $S$ be a finite set and $\mathcal F$ be a subset of $2^S$.
Let $\mathcal A(\tau)$ be an algorithm that outputs all sets in $\mathcal F$ with the cardinality at least $\tau$.
We denote the maximum delay complexity and the space complexity from $\mathcal A(\tau)$ to $\mathcal A(1)$ as $t(n)$ and $s(n)$, respectively.
Moreover, we denote the number of outputs of $\mathcal A(\tau)$ as $\#\mathcal A(\tau)$.
Under this problem setting, 
we construct a ranked enumeration algorithm that outputs the $i$-th solution in $\order{i\cdot n\cdot t(n)}$ time with $\order{s(n)}$ space as follows.

Our idea is simply to execute $\mathcal A$ from $\mathcal A(n)$ to $\mathcal A(1)$.
When $\mathcal A(k)$ outputs a solution with cardinality more than $k$, we just ignore it.
In other words, when we execute $\mathcal A(i)$, all solutions in $\mathcal F$ with cardinality exactly $i$ are output.
Clearly, we can enumerate all solutions in $\mathcal F$ in a non-increasing order of their cardinality.
We consider the time and space complexity of this method.
It is easy to see that the space complexity of this algorithm is $\order{s(n)}$ as we just execute $\mathcal A$ in order.
Thus, we estimate the running time required to output the first $i$ solutions for $i \le \size{\mathcal F}$.
Let $j \ge 1$ be the maximum integer such that $\#\mathcal A(j)$ is less than $i$.
Since the delay of $\mathcal A$ is bounded by $t(n)$ and $\#\mathcal A(j-1)$ is at least $i$,
$\mathcal A(j-1)$ outputs the $i$-th solution in $\order{i\cdot t(n)}$ time.
Since the total running time is bounded by $\order{(n-j+1)\cdot\#\mathcal A(j)\cdot t(n) + i\cdot t(n)} = \order{i\cdot n\cdot t(n)}$ time, 
this algorithm outputs the first $i$ solutions in $\order{i\cdot n\cdot t(n)}$ time.

\begin{theorem}\label{thm:top}
    Let $S$ be a finite set and $\mathcal F$ be a subset of $2^S$.
    For any $1 \le k \le \tau$, suppose that we have 
    an algorithm $\mathcal A(k)$ that enumerates 
    all sets in $\mathcal F$ with the cardinality at least $k$ for any $1 \le k \le \tau$ in $t(n)$ delay and $s(n)$ space.
    Then, there is an algorithm enumerating all subsets in $\mathcal I$ in non-increasing order of their cardinality that outputs the first $i$ solutions in $\order{i\cdot n \cdot t(n)}$ time using $\order{s(n)}$ space for $i \leq \size{\mathcal F}$.
\end{theorem}

We obtain a linear incremental-time and polynomial-space ranked enumeration algorithm for maximal common independent sets of two matroids by combining \Cref{thm:top,thm:large}.

\begin{theorem}\label{thm:ranked}
    There is a linear incremental-time and polynomial-space algorithm for enumerating 
    all maximal common independent sets in two matroids in non-increasing order.
    This algorithm outputs the first $i$ solutions in $\order{i\cdot n^7Q}$ time.
\end{theorem}

\begin{theorem}\label{thm:parity:ranked}
    For a tractable pair,
    there is a linear incremental-time and polynomial-space algorithm for enumerating 
    all maximal matchings of a matroid in non-increasing order.
\end{theorem}

\section{Applications of our algorithms}\label{sec:application}
Due to an expressive power of \textsc{Matroid Intersection} and \textsc{Matroid Matching},
\Cref{thm:large,thm:ranked,thm:parity:large,thm:parity:ranked} give enumeration algorithms for various combinatorial objects in a unified way.
An example of such objects is a $b$-matchings in bipartite graphs.
In this section, we present several combinatorial objects represented by an intersection of basic matroids or a matching of a tractable pair.
We would like to note that the following facts are well-known in the literature but we concisely explain them for self-containedness.

Let $G = (V, E)$ be an undirected graph and $\mathcal F$ be the set of forests in $G$.
It is known that $(E, \mathcal F)$ is a matroid, called a \emph{graphic matroid}.
Let $S$ be a finite set, $\mathcal S = \set{S_1, \dots, S_k}$ be a partition of $S$, and 
$s_1, \dots s_k$ be $k$ integers.
We define $\mathcal I \coloneqq \inset{T \subseteq S}{\size{T \cap S_i} \le s_i \text{ for each } 1 \le i \le k}$.
Then, it is known that $(S, \mathcal I)$ is a matroid, called a \emph{partition matroid}. 
See \cite{book:oxley} for more details.

Let $G = (X, Y, E)$ be a bipartite graph with color classes $X$ and $Y$.
Let $b \colon X \cup Y \to \mathbb N$.
An edge subset $M \subseteq E$ is called a \emph{$b$-matching} of $G$ if for every $v \in X \cup Y$, the number of edges in $M$ incident to $v$, denoted $d_M(v)$, is at most $b(v)$.
When $b(v) = 1$ for $v \in X \cup Y$, $M$ is indeed a matching of $G$.
We define two partition matroids $\mat_X = (E, \mathcal I_X)$ and $\mat_Y = (E, \mathcal I_Y)$ as
\begin{align*}
    \mathcal I_X &= \set{M \subseteq E : d_M(v) \le b(v) \text{ for } v \in X}\\
    \mathcal I_Y &= \set{M \subseteq E : d_M(v) \le b(v) \text{ for } v \in Y}.
\end{align*}
Thus, the set of all $b$-matchings of $G$ is defined as the set of all common independent sets of $\mat_X$ and $\mat_Y$.

Let $G = (V, E)$ be an edge-colored undirected graph and $E_i$ be the set of edges with color $i$.
An edge subset $F$ of $G$ is called a \emph{colorful forest} if it is a forest and edges in it have distinct colors.
We consider two matroids $\mat_G = (E, \mathcal I_G)$ and $\mat_c = (E, \mathcal I_c)$ as
\begin{align*}
    \mathcal I_G &= \set{F \subseteq E : F \text{ is a forest in } G}\\
    \mathcal I_c &= \set{F \subseteq E : \size{F \cap E_i} \le 1}.
\end{align*}
Thus, the set of all colorful forests of $G$ is defined as the set of all common independent sets of $\mat_G$ and $\mat_c$.

Let $D = (V, A)$ be a directed graph and let $\delta^+$ and $\delta^-$ be functions from $V$ to $\mathbb N$.
An edge subset $F$ of $D$ is called a \emph{degree constrained subgraph} (for $\delta^+$ and $\delta^-$) of $G$ if for every $v \in V$, it holds that $d^+_F(v) \le \delta^+(v)$ and $d^-_F(v) \le \delta^-(v)$, where $d^+_F(v)$ (resp. $d^-_F(v)$) is the number of outgoing arcs from $v$ (resp. incoming arcs to $v$) in $F$.
We define two matroids $\mat_+ = (A, \mathcal I_+)$ and $\mat_- = (A, \mathcal I_-)$ as
\begin{align*}
    \mathcal I_+ &= \set{F \subseteq A : d_F^+(v) \le \delta^+(v) \text{ for each } v \in V}\\
    \mathcal I_- &= \set{F \subseteq A : d_F^-(v) \le \delta^-(v) \text{ for each } v \in V}.
\end{align*}
Thus, the set of all degree constrained subgraphs of $D$ is defined as the set of all common independent sets of $\mat_+$ and $\mat_-$.

For the above matroids, we can easily construct polynomial-time independence oracles.

\begin{theorem}\label{thm:app}
    There are polynomial-delay and polynomial-space enumeration algorithms for 
    \begin{itemize}
        \item maximal bipartite $b$-matchings with cardinality at least $\tau$,
        \item maximal colorful forests with cardinality at least $\tau$, and
        \item maximal degree constrained subgraphs in digraphs with cardinality at least $\tau$,
    \end{itemize}
    Moreover, there are linear incremental-time and polynomial-space ranked enumeration algorithms for the above problems.
\end{theorem}


As an application of \Cref{thm:parity:large}, we give a polynomial-delay algorithm for enumerating minimal connected vertex covers with cardinality at most $\tau$ in subcubic graphs by transforming the instance into that of \textsc{Large Maximal Matroid Matching Enumeration}.
This reduction is presented in~\cite{UENO1988355}, where they gave a polynomial-time algorithm for finding a minimum connected vertex cover in subcublic graphs.
To be self-contained, we give a bijection between the set of minimal connected vertex covers in a subcubic graph $G$ and the set of maximal matchings of a tractable pair $(\mat, H)$.


Let $G = (V, E)$ be a connected subcubic graph.
We assume that $G$ has at least two verticces, as otherwise the problem is trivial.
For a graph, a vertex set is said to be \emph{connected} if it induces a connected subgraph in it.
In the following, we rather consider the collection of the complements of minimal connected vertex covers of $G$.
It is easy to see that $X \subseteq V$ is a (minimal) connected vertex cover of $G$ if and only if $V\setminus X$ is a (maximal) non-separating independent set of $G$, that is, $V \setminus X$ is an independent set of $G$ and $G[X]$ is connected.
We construct the following graph $G' = (V', E')$, which may have self-loops and parallel edges.
We start with $G' = G$ and, for each vertex $v$ of $G$, we do the following.
\begin{enumerate}
    \item If $v$ has exactly three incident edges $e_1, e_2, e_3$,
    we replace $v$ with three vertices $v_1$, $v_2$ and $v_3$ in such a way that each $v_i$ has $e_i$ as an incident edge.
    We also add two edges $f^v_1 = \set{v_1, v_2}$ and $f^v_2 = f\set{v_2, v_3}$ to $G'$.
    \item If $v$ has exactly two incident edges $e_1$ and $e_2$, 
    we replace $v$ with two vertices $v_1$ and $v_2$ in such a way that each $v_i$ has $e_i$ as an incident edge.
    We also add two parallel edges $f^v_1$ and $f^v_2$ between $v_1$ and $v_2$.
    \item If $v$ has only one incident edge $e = \set{u, v}$,
    we add two self-loops $f^v_1$ and $f^v_2$ incident to $v$.
    In this case, we set $v_1 = v$.
\end{enumerate}
Moreover, for each (original) edge $e \in E$, we subdivide it into a path with length $2$.
The two edges in the path are denoted by $e'$ and $e''$. 
See \Cref{fig:construct} for a concrete example.
For the resultant graph $G'$, we consider the \emph{cographic matroid} $\mat$ of $G'$, that is, $\mat = (E', \mathcal I_{G'})$ such that, for $F \subseteq E'$, $F \in \mathcal I_{G'}$ if and only if $G' - F$ is connected.

\begin{figure}[t]
    \centering
    \includegraphics[width=0.6\textwidth]{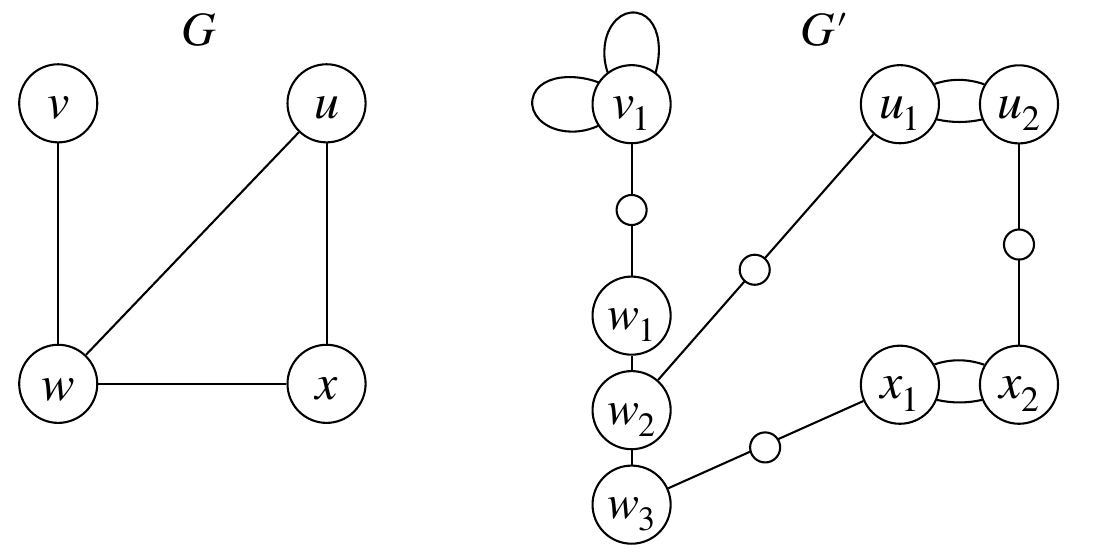}
    \caption{The construction of a subcubic graph $G'$ from a subcubic graph $G$.
    The set of non-separating independent sets of $G$ corresponds to the set of maximal cographic matroid matching of $G'$.}
    \label{fig:construct}
\end{figure}

We next define another graph $H = (V_H, E_H)$.
The vertex set of $H$ corresponds to $V_H = E'$ and the edge set corresponds to $E_H = \set{\set{f^v_1, f^v_2} : v \in V} \cup \set{\set{e', e''} : e \in E}$.
Since every cographic matroid is linear, the pair $(\mat, H)$ is tractable.
We show that each a matching of $\mathbf M$ corresponds to a non-separating independent set of $G$ and vice-versa.
From the definition of $G'$, there is a bijection $\phi$ between the edges in $H$ and $V \cup E$: for $e_H \in E_H$
\begin{align*}
    \phi(e_H) = \begin{dcases}
          v & \text{if } e_H = \{f^v_1, f^v_2\}\\
          e & \text{if } e_H = \{e', e''\}
    \end{dcases}.
\end{align*}
For $M \subseteq E_H$, we let $\phi(M) = \bigcup_{e_H \in M} \phi(e_H)$.
The following lemma is the heart of our reduction.

\begin{lemma}\label{lem:app:matching-nis}
    Let $M \subseteq E_H$.
    If $M$ is a matching of $(\mat, H)$, then we have $\phi(M) \subseteq V$.
    Moreover, suppose that $\phi(M) \subseteq V$.
    Then, $M$ is a matching of $(\mat, H)$ if and only if $\phi(M)$ is a non-separating independent set of $G$.
\end{lemma}
\begin{proof}
    Suppose that $M$ is a matching of $(\mat, H)$.
    Observe that $M$ has no edge $e_H \in E_H$ with $\phi(e_H) = e$ for some $e \in E$.
    This follows from the fact that by removing $e'$ and $e''$ from $G'$, the central vertex becomes an isolated vertex, which implies that $e_H = \set{e', e''}$ is dependent in the cographic matroid $\mat$.
    Thus, we have $\phi(M) \subseteq V$.

    To prove the latter claim, suppose that $\phi(M) \subseteq V$.
    We first show that $\phi(M)$ is an independent set of $G$ for each matching $M$ of $(\mat, H)$.
    Let $u$ and $v$ be two adjacent vertices in $G$.
    Suppose for contradiction that $M$ contains two edges
    $\phi^{-1}(u) = \set{f^u_1, f^u_2} \in E_H$ and $\phi^{-1}(v) = \set{f^v_1, f^v_2} \in E_H$.
    Since each edge in $G$ is replaced with a path of length $2$ in $G'$, $G'$ has a path of length $2$ between $u_i$ and $v_j$ for some $i$ and $j$.
    We let the central vertex on the path be $w$.
    By the construction of $G'$, the degrees of $u_i$ and $v_j$ are $1$ in $G' - \set{f^u_1, f^u_2, f^v_1, f^v_2}$, and $u$ and $v$ are both adjacent to $w$.
    This implies that $\set{f^u_1, f^u_2, f^v_1, f^v_2}$ is dependent in $\mat$, a contradiction.

    We next show that $\phi(M)$ is non-separating in $G$, that is, $G[V \setminus \phi(M)]$ is connected.
    Suppose for contradiction that $G[V \setminus \phi(M)]$ is disconnected.
    Let $u$ and $v$ be two vertices in $V \setminus \phi(M)$ such that they belong to distinct connected components in $G[V \setminus \phi(M)]$.
    For $w \in V$, we denote by $V'(w)$ the set of (at most three) vertices in $G'$ corresponds to $w$.
    Let $P'$ be an arbitrary path between $V'(u)$ and $V'(v)$ in $G'$ such that all internal vertices are disjoint from $V'(u) \cup V'(v)$.
    We claim that $P'$ contains at least one edge $f^w_i$ such that $\set{f^w_1, f^w_2} \in M$, meaning that $V(M)$ is dependent in $\mat$. 
    From $P'$, we can naturally construct a path $P$ between $u$ and $v$ in $G$, by taking all edges $e \in E$ such that $\set{e', e''}$ is contained in $P'$.
    As $u$ and $v$ are not connected in $G[V \setminus \phi(M)]$, at least one vertex $w \neq u, v$ of $P$ is contained in $\phi(M)$.
    This implies that $P'$ contains at least two vertices from $V'(w)$.
    As $P'$ contains at least one of $f^w_1$ and $f^w_2$, this leads a contradiction.


    Conversely, suppose that $\phi(M)$ is a non-separating independent set of $G$.
    We claim that $M$ is a matching of $(\mat, H)$.
    To see this, we show the following two claims:
    \begin{enumerate}
        \item For any pair of distinct vertices $u, v \in V \setminus \phi(M)$, $G' - \bigcup_{e_H \in M}e_H$ has a path between a vertex in $V'(u)$ and a vertex in $V'(v)$; and
        \item For any $u \in \phi(M)$, there is a vertex $v \in V \setminus \phi(M)$ such that 
        $G' - \bigcup_{e_H \in M} e_H$ has a path between $u_i$ and a vertex in $V'(v)$ for any $u_i \in V'(u)$.
    \end{enumerate}
    Let us note that $G[V'(w)]$ is connected if $w \notin V \setminus \phi(M)$.
    These two claims together imply that $G' - \bigcup_{e_H \in M} e_H$ is connected, meaning that $M$ is a matching of $(\mat, H)$.

    The first claim is verified as follows.
    Since $I$ is a non-separating independent set  of $G$, for $u, v \in V \setminus I$, 
    $G[V \setminus I]$ has a path $P$ between $u$ and $v$. 
    Thus, $G' - \bigcup_{e_H \in M}e_H$ also has a path $P'$ between a vertex in $V'(u)$ and a vertex $V'(v)$ by taking all edges of the preimage of $P$ under $\phi$.

    The second claim is verified as follows.
    Let $u \in \phi(M)$ and let $v \in V \setminus \phi(M)$.
    Since $\phi(M)$ is an independent set of $G$,
    $u$ has a neighbor $w \notin I$.
    As $\phi(M) \subseteq V$, there is a path corresponding to the edge $\set{u, w}$ in $G' - \bigcup_{e_H \in M} e_H$.
    Moreover, by the first claim, $G' - \bigcup_{e_H \in M} e_H$ has a path between $w_i$ and $v_j$ for any $w_i \in V'(w)$ and $v_j \in V'(v)$, which proves the second claim by concatenating these paths. 
 \end{proof}

As observed in the proof of \Cref{lem:app:matching-nis}, every matching of $(\mat, H)$ does not contain any edge $e_H$ with $\phi(e_H) = \set{e', e''}$, the non-separating independent set $\phi(M) \cap V$ of $G$ is uniquely determined and vice-versa.
Thus, it suffices to enumerate all maximal matchings of $(\mat, H)$ to enumerate all maximal non-separating independent sets of $G$ (or equivalently minimal connected vertex covers of $G$).
It is known that every cographic matroid $\mat$ is linear and a matrix representation of $\mat$ can be constructed in polynomial time~\cite{book:oxley}.
Therefore, $(\mat, H)$ is tractable~\cite{Gabow1986}, and
we obtain the following theorem.

\begin{theorem}\label{thm:app:matroid}
    There is a polynomial-delay and polynomial-space 
    algorithm that, given a subcubic graph $G$ and an integer $\tau$, enumerates all minimal connected vertex covers of $G$ with cardinality at most $\tau$.
    Moreover, there is linear incremental-time and polynomial-space ranked enumeration algorithms for the above problem.
\end{theorem}

\bibliographystyle{abbrv}
\bibliography{main.bib}

\begin{thebibliography}{10}

\bibitem{Avis::1996}
D.~Avis and K.~Fukuda.
\newblock Reverse search for enumeration.
\newblock {\em Discret. Appl. Math.}, 65(1):21 -- 46, 1996.

\bibitem{Birmele:SODA:2013}
E.~Birmel^^c3^^a9, R.~Ferreira, R.~Grossi, A.~Marino, N.~Pisanti, R.~Rizzi, and
  G.~Sacomoto.
\newblock Optimal listing of cycles and st-paths in undirected graphs.
\newblock In {\em Proc. of {SODA} 2013}, pages 1884--1896. {SIAM}, 2013.

\bibitem{Boros:MFCS:2002}
E.~Boros, K.~Elbassioni, V.~Gurvich, and L.~Khachiyan.
\newblock Matroid intersections, polymatroid inequalities, and related
  problems.
\newblock In {\em Proc. of {MFCS} 2002}, pages 143--154, Berlin, Heidelberg,
  2002. Springer Berlin Heidelberg.

\bibitem{capelli_et_al:LIPIcs.STACS.2023.18}
F.~Capelli and Y.~Strozecki.
\newblock {Geometric Amortization of Enumeration Algorithms}.
\newblock In {\em Proc. of {STACS} 2023}, volume 254 of {\em LIPIcs}, pages
  18:1--18:22, Dagstuhl, Germany, 2023. Schloss Dagstuhl -- Leibniz-Zentrum
  f{\"u}r Informatik.

\bibitem{DBLP:conf/fct/CaselFGMS19}
K.~Casel, H.~Fernau, M.~K. Ghadikolaei, J.~Monnot, and F.~Sikora.
\newblock Extension of some edge graph problems: Standard and parameterized
  complexity.
\newblock In {\em Proc. of {FCT} 2019}, volume 11651 of {\em LNCS}, pages
  185--200. Springer, 2019.

\bibitem{DBLP:conf/spire/ConteGMUV17}
A.~Conte, R.~Grossi, A.~Marino, T.~Uno, and L.~Versari.
\newblock Listing maximal independent sets with minimal space and bounded
  delay.
\newblock In {\em Proc. of {SPIRE} 2017}, volume 10508 of {\em LNCS}, pages
  144--160. Springer, 2017.

\bibitem{doi:10.1137/0215066}
W.~H. Cunningham.
\newblock Improved bounds for matroid partition and intersection algorithms.
\newblock {\em SIAM J. Comput.}, 15(4):948--957, 1986.

\bibitem{Deep:ICDT:2021}
S.~Deep and P.~Koutris.
\newblock {Ranked Enumeration of Conjunctive Query Results}.
\newblock In {\em Proc. of {ICDT} 2021}, volume 186 of {\em LIPIcs}, pages
  5:1--5:19, Dagstuhl, Germany, 2021. Schloss Dagstuhl -- Leibniz-Zentrum
  f{\"u}r Informatik.

\bibitem{edmonds_1965}
J.~Edmonds.
\newblock Paths, trees, and flowers.
\newblock {\em Can. J. Math.}, 17:449^^e2^^80^^93467, 1965.

\bibitem{edmonix1969submodular}
J.~Edmonds.
\newblock {\em Submodular Functions, Matroids, and Certain Polyhedra}, pages
  11--26.
\newblock Springer Berlin Heidelberg, Berlin, Heidelberg, 2003.

\bibitem{Eppstein:k-best:2016}
D.~Eppstein.
\newblock {$k$-Best Enumeration}.
\newblock In {\em Encyclopedia of Algorithms}, pages 1003--1006. Springer, New
  York, NY, 2016.

\bibitem{FUKUDA1995231}
K.~Fukuda and M.~Namiki.
\newblock Finding all common bases in two matroids.
\newblock {\em Discret. Appl. Math.}, 56(2):231--243, 1995.
\newblock Fifth Franco-Japanese Days.

\bibitem{Gabow1986}
H.~N. Gabow and M.~Stallmann.
\newblock {An augmenting path algorithm for linear matroid parity}.
\newblock {\em Combinatorica}, 6(2):123--150, 1986.

\bibitem{doi:10.1137/0211014}
P.~M. Jensen and B.~Korte.
\newblock Complexity of matroid property algorithms.
\newblock {\em SIAM J. Comput.}, 11(1):184--190, 1982.

\bibitem{Johson:Yanakakis:Papadimitriou:IPL:1988}
D.~S. Johnson, M.~Yannakakis, and C.~H. Papadimitriou.
\newblock {On generating all maximal independent sets}.
\newblock {\em Inform. Process. Lett.}, 27(3):119 -- 123, 1988.

\bibitem{DBLP:conf/wads/KanteLMNU15}
M.~M. Kant{\'{e}}, V.~Limouzy, A.~Mary, L.~Nourine, and T.~Uno.
\newblock {Polynomial Delay Algorithm for Listing Minimal Edge Dominating Sets
  in Graphs}.
\newblock In {\em Proc. of {WADS} 2015}, pages 446--457, 2015.

\bibitem{10.1007/11841036_41}
L.~Khachiyan, E.~Boros, K.~Borys, K.~Elbassioni, V.~Gurvich, and K.~Makino.
\newblock Enumerating spanning and connected subsets in graphs and matroids.
\newblock In {\em Proc. of {ESA} 2006}, volume 4168 of {\em LNCS}, pages
  444--455, Berlin, Heidelberg, 2006. Springer Berlin Heidelberg.

\bibitem{DBLP:journals/algorithmica/KhachiyanBBEGM08}
L.~Khachiyan, E.~Boros, K.~Borys, K.~M. Elbassioni, V.~Gurvich, and K.~Makino.
\newblock {Generating Cut Conjunctions in Graphs and Related Problems}.
\newblock {\em Algorithmica}, 51(3):239--263, 2008.

\bibitem{DBLP:journals/siamdm/KhachiyanBEGM05}
L.~G. Khachiyan, E.~Boros, K.~M. Elbassioni, V.~Gurvich, and K.~Makino.
\newblock On the complexity of some enumeration problems for matroids.
\newblock {\em {SIAM} J. Discret. Math.}, 19(4):966--984, 2005.

\bibitem{Kobayashi:arXiv:2020}
Y.~Kobayashi, K.~Kurita, and K.~Wasa.
\newblock Efficient constant-factor approximate enumeration of minimal subsets
  for monotone properties with weight constraints.
\newblock {\em CoRR}, abs/2009.08830, 2020.

\bibitem{Kobayashi:WG:2022}
Y.~Kobayashi, K.~Kurita, and K.~Wasa.
\newblock Polynomial-delay and polynomial-space enumeration of large maximal
  matchings.
\newblock In {\em Proc. of {WG} 2022}, volume 13453 of {\em LNCS}, pages
  342--355, Cham, 2022. Springer International Publishing.

\bibitem{kobayashi_et_al:LIPIcs.MFCS.2023.58}
Y.~Kobayashi, K.~Kurita, and K.~Wasa.
\newblock {Polynomial-Delay Enumeration of Large Maximal Common Independent
  Sets in Two Matroids}.
\newblock In {\em Proc. of {MFCS} 2023}, volume 272 of {\em LIPIcs}, pages
  58:1--58:14, Dagstuhl, Germany, 2023. Schloss Dagstuhl -- Leibniz-Zentrum
  f{\"u}r Informatik.

\bibitem{Korhonen:arXiv:2020}
T.~Korhonen.
\newblock Listing small minimal separators of a graph.
\newblock {\em CoRR}, abs/2012.09153, 2020.

\bibitem{DBLP:conf/mfcs/KuritaK20}
K.~Kurita and Y.~Kobayashi.
\newblock {Efficient Enumerations for Minimal Multicuts and Multiway Cuts}.
\newblock In {\em Proc. of {MFCS} 2020}, volume 170 of {\em LIPIcs}, pages
  60:1--60:14, Dagstuhl, Germany, 2020. Schloss Dagstuhl -- Leibniz-Zentrum
  f{\"u}r Informatik.

\bibitem{Lawler1975}
E.~L. Lawler.
\newblock {Matroid intersection algorithms}.
\newblock {\em Math. Program.}, 9(1):31--56, 1975.

\bibitem{Lawler:Lenstra:Rinnooy:SIAM:1980}
E.~L. Lawler, J.~K. Lenstra, and A.~H.~G. Rinnooy~Kan.
\newblock {Generating All Maximal Independent Sets: {NP}-Hardness and
  Polynomial-Time Algorithms}.
\newblock {\em SIAM J. Comput.}, 9(3):558--565, 1980.

\bibitem{DBLP:conf/swat/MakinoU04}
K.~Makino and T.~Uno.
\newblock New algorithms for enumerating all maximal cliques.
\newblock In {\em {Proc.} {SWAT} 2004}, volume 3111 of {\em LNCS}, pages
  260--272. Springer, 2004.

\bibitem{book:oxley}
J.~Oxley.
\newblock {\em Matroid theory}, volume~3.
\newblock Oxford University Press, USA, 2006.

\bibitem{DBLP:conf/pods/RavidMK19}
N.~Ravid, D.~Medini, and B.~Kimelfeld.
\newblock {Ranked Enumeration of Minimal Triangulations}.
\newblock In {\em Proc. of {PODS} 2019}, pages 74--88, 2019.

\bibitem{Tarjan:1975}
R.~C. Read and R.~E. Tarjan.
\newblock Bounds on backtrack algorithms for listing cycles, paths, and
  spanning trees.
\newblock {\em Networks}, 5(3):237--252, 1975.

\bibitem{DBLP:journals/siamcomp/TsukiyamaIAS77}
S.~Tsukiyama, M.~Ide, H.~Ariyoshi, and I.~Shirakawa.
\newblock {A New Algorithm for Generating All the Maximal Independent Sets}.
\newblock {\em {SIAM} J. Comput.}, 6(3):505--517, 1977.

\bibitem{UENO1988355}
S.~Ueno, Y.~Kajitani, and S.~Gotoh.
\newblock On the nonseparating independent set problem and feedback set problem
  for graphs with no vertex degree exceeding three.
\newblock {\em Discrete Math.}, 72(1):355--360, 1988.

\bibitem{10.1007/3-540-63890-3_11}
T.~Uno.
\newblock Algorithms for enumerating all perfect, maximum and maximal matchings
  in bipartite graphs.
\newblock In {\em Proc. of {ISSAC} 1997}, volume 1350 of {\em LNCS}, pages
  92--101, Berlin, Heidelberg, 1997. Springer Berlin Heidelberg.

\bibitem{10.1007/3-540-45678-3_32}
T.~Uno.
\newblock A fast algorithm for enumerating bipartite perfect matchings.
\newblock In {\em Proc. of {ISSAC} 2001}, volume 2223 of {\em LNCS}, pages
  367--379, Berlin, Heidelberg, 2001. Springer Berlin Heidelberg.

\end{thebibliography}

\end{document}